\documentclass[12pt]{article}

\usepackage{fontenc,fontaxes}
\usepackage{amsthm}
\usepackage{amsmath,amsfonts}
\usepackage{authblk}
\usepackage{natbib}
\usepackage{comment}
\usepackage{mathtools}
\usepackage{arydshln}
\usepackage{bbm}
\usepackage[bf,small]{caption}
\usepackage{xcolor}
\usepackage[colorlinks = true,
linkcolor = black,
urlcolor  = black,
citecolor = black,
anchorcolor = black]{hyperref}

\newtheorem{lemma}{Lemma}[section]
\newtheorem{proposition}{Proposition}[section]

\theoremstyle{definition}
\newtheorem{remark}{Remark}[section]

\newtheorem*{assumption*}{Assumption}

\newtheorem{example}{Example}[section]

\mathtoolsset{mathic=true}
\DeclareMathOperator{\clr}{clr}
\DeclareMathOperator{\ran}{ran}
\DeclareMathOperator{\cl}{cl}
\DeclareMathOperator{\id}{id}
\DeclareMathOperator{\spn}{span}
\newcommand{\zero}{\mathbf{0}}
\numberwithin{equation}{section}
\numberwithin{table}{section}

\renewcommand\footnotemark{}

\begin{document}
	
	\title{Cointegrated Density-Valued Linear Processes}
	\author{Won-Ki Seo}
	
	\thanks{\noindent I thank Brendan K.\ Beare for a lot of discussion at all stages of this paper. I am also grateful to Dakyung Seong and Lam Nguyen for useful discussion on section \ref{empirical} of this paper.}
	\affil{Department of Economics, University of California, San Diego}
	
	\maketitle
	\bigskip
	
	\begin{abstract}
		In data rich environments we may sometimes deal with time series that are probability density-function valued, such as observations of cross-sectional income distributions over time. To apply the methods of functional time series analysis to such observations, we should first embed them in a linear space in which the essential properties of densities are preserved under addition and scalar multiplication. Bayes Hilbert spaces provide one way to achieve this embedding. In this paper we investigate the use of Bayes Hilbert spaces to model cointegrated density-valued linear processes. We develop an I(1) representation theory for cointegrated linear processes in a Bayes Hilbert space, and adapt existing statistical procedures for estimating the corresponding attractor space to a Bayes Hilbert space setting. We revisit empirical applications involving earnings and wage densities to illustrate the utility of our approach.
	\end{abstract}
	
	\pagebreak
	
	\section{Introduction}\label{intro}
	While the subject of time series analysis has traditionally dealt with time series taking values in finite dimensional Euclidean space, a recent literature on functional time series analysis deals with time series taking values in an infinite dimensional Banach or Hilbert space. Each observation of such a time series is a functional object; for example, it could be a square-integrable function, continuous function, or probability density function. \cite{Bosq2000} gives a rigorous theoretical treatment of linear processes in Banach and Hilbert spaces. \cite{HK2012} discuss statistical aspects of functional data and time series analysis, and provide various empirical applications. 
	
	\cite{granger1981} introduced the notion of cointegration as a way to model long-run equilibrium relationships between real-valued time series. A recent paper by \cite{Chang2016152} is the first to consider the possibility of cointegration in a functional time series setting. The authors consider a time series of probability densities taking values in the space of square-integrable real functions on a compact interval $K$, denoted by $L^2(K)$, and provide a notion of cointegration adapted to this space. They develop associated statistical methods based on functional principal component analysis (FPCA), and provide empirical applications to time series of earnings and stock return densities.
	
	\cite{Beare2017} notes that a technical complication arises in the framework developed by \cite{Chang2016152}: the nonnegativity property of probability densities is incompatible with the type of nonstationarity exhibited by integrated time series, and this incompatibility cannot be resolved by a simple demeaning of densities. Nontrivial examples of cointegrated density-valued processes in $L^2(K)$ therefore do not exist. Even in models of stationary probability density-valued time series, \cite{petersen2016} have observed that it is generally inadvisable to treat such time series as taking values in the subset of $L^2(K)$ consisting of probability densities, as this subset does not form a linear subspace of $L^2(K)$. It is clear that an arbitrary linear combination of densities is not a proper density; only a convex combination of densities is a proper one. 
	
	The primary purpose of this paper is to show that the technical complications just discussed can be resolved by viewing probability density-valued observations as elements of a Bayes Hilbert space. Such spaces were introduced by \cite{Egozcue2006} and developed further by \cite{Boogaart2014}. They are constructed in such a way that different elements of the space correspond to different probability densities, and the essential properties of probability densities are preserved under addition and scalar multiplication. The notion of cointegration may easily be adapted to this space as in \cite{BSS2017}, who extend the framework developed by \cite{Chang2016152} to arbitrary complex separable Hilbert spaces.
	
	A secondary purpose of the paper is to study the behavior of density-valued autoregressive processes of order $p$  (AR($p$) processes) taking values in a Bayes Hilbert space. We provide conditions under which an AR($p$) law of motion has a stationary solution or I($d$) solution for some positive integer $d$, and a necessary and sufficient condition for such an I($d$) solution to in fact be I(1). These results are closely related to the so-called Granger-Johansen representation theorem and its generalization to a possibly infinite dimensional Hilbert space setting by \cite{BSS2017} and \cite{BS2017}.  
	
	A third contribution of the paper is the provision of statistical methods to estimate the attractor space (to be defined later) for an I(1) process in a Bayes Hilbert space. These methods are based on the functional unit root test proposed by \cite{Chang2016152}. We illustrate their usefulness with empirical applications to time series of cross-sectional densities of individual earnings and wages.  
	
	The remainder of the paper is organized as follows. In Section \ref{prelim}, we review some background material on Bayes Hilbert spaces and other essential mathematical concepts. In Section \ref{coden}, we explain how Bayes Hilbert spaces can provide a useful setting for models of cointegrated density-valued time series. Results on I($d$) representations of cointegrated AR($p$) processes in Bayes Hilbert space are provided in Section \ref{arden}. Statistical tools for studying cointegrated density-valued time series are introduced in Section \ref{stat}, and illustrated with empirical applications in Section \ref{empirical}.

	\section{Preliminaries}\label{prelim}
	Here we briefly review essential background for the study of cointegrated density-valued linear processes, and fix standard notation and terminology. 
	
	\subsection{Bayes Hilbert spaces}\label{bayessp}
	Bayes Hilbert spaces provide the setting for our treatment of cointegrated density-valued linear processes. The discussion provided here will omit some details to conserve space. The reader is referred to \cite{Egozcue2006} and  \cite{Boogaart2014} for a rigorous introduction to the subject. 
	
	Let $(M, \mathcal A)$ be a measurable space and $\mathcal P$ be the set of $\sigma$-finite positive real-valued measures defined on it. For any measure $\lambda$, called a reference measure, define 
	\begin{align*}
	\mathcal M(\lambda) \coloneqq \{ \nu \in \mathcal P :  \nu(A) = 0 \, \Leftrightarrow \, \lambda(A)=0,  \quad \forall A\in \mathcal A\}.
	\end{align*} 
	That is, $\mathcal M(\lambda)$ is the collection of $\sigma$-finite positive measures absolutely continuous with respect to $\lambda$. Due to the Radon-Nikodym theorem, any element $\nu$ in $\mathcal M(\lambda)$ may be identified with the $\lambda$-density $\mathrm{d}\nu/ \mathrm{d}\lambda$, so hereafter we always regard an element in $\mathcal M(\lambda)$ as its $\lambda$-density.   
	
	We define an equivalence relation $\simeq$ on $\mathcal M(\lambda)$ as follows. For $\lambda$-densities $f, g \in \mathcal M(\lambda)$, the equivalence $f\simeq g$ holds if there is some $a>0$ such that
	\begin{align*}
	\int_{A} f\mathrm{d}\lambda &= a\cdot\int_{A} g\mathrm{d}\lambda \quad\text{for all }A \in \mathcal A.
	\end{align*}
	The above says that if two measures $\nu_1$ and $\nu_2$ are proportional in the sense that $\nu_1(A) = a \cdot \nu_2 (A)$ for all $A \in \mathcal A$, their $\lambda$-densities are equivalent under $\simeq$. The collection of  $\simeq$-equivalence classes of $\lambda$-densities, denoted by $B(\lambda)$, are the elements of a Bayes space associated with $\lambda$. We define two vector operations $(\oplus, \odot)$, called perturbation and powering, as follows.  For $a \in \mathbb{R}$ and $f,g \in B(\lambda)$, define
	\begin{equation*}
	f \oplus g \simeq \int fg\mathrm{d}\lambda,
	\end{equation*} 
	\begin{equation*}
	a \odot f \simeq \int f^{\,a}\mathrm{d}\lambda.
	\end{equation*} 
	Negative perturbation $\ominus$ is defined as $f \ominus g = f \oplus (-1 \odot g) $.  $B(\lambda)$ equipped with the two vector operations $(\oplus, \odot)$ is a vector space. 
	
	We hereafter assume that $\lambda$ is a finite measure and let   $B^2(\lambda)$ be the collection of $\lambda$-densities that have square-integrable logarithms, i.e.
	\begin{align*}
	B^2(\lambda) \coloneqq \left\{ f \in B(\lambda) \,:\, \int \left| \log f\right|^2\mathrm{d}\lambda < \infty    \right\}.
	\end{align*} 
	It turns out that $B^2(\lambda)$ is a vector subspace of $B(\lambda)$. Define the inner product $\langle \cdot, \cdot \rangle_{B^2(\lambda)}$ in the following way: for $f,g \in B^2(\lambda)$,
	\begin{align*}
	\langle f , g \rangle_{B^2(\lambda)} = \int \left(\log f - \int \log f\mathrm{d}\lambda \right) \left(\log g -  \int \log f \mathrm{d}\lambda \right)\mathrm{d}\lambda.
	\end{align*}
	The following result is obtained in \cite{Boogaart2014}. 
	\begin{lemma}\label{bsplem1}
		$B^2(\lambda)$ equipped with $(\oplus,\odot)$  and $\langle \cdot , \cdot \rangle_{B^2(\lambda)} $ is a separable Hilbert space. 
	\end{lemma}
	
	Define $L^2(\lambda)$ as the set of measurable square-integrable functions $f : \Omega \to \mathbb{R}$. An element of $L^2(\lambda)$ is regarded as a class of square-integrable functions that are $\lambda$-almost surely equivalent. $L^2(\lambda)$ is assumed to be equipped with the inner product $\langle \cdot , \cdot \rangle_{L^2(\lambda)}$ defined as follows: $\langle \tilde{f},\tilde{g} \rangle_{L^2(\lambda)} = \int \tilde{f} \tilde{g}\mathrm{d}\lambda$ for $\tilde{f}, \tilde{g} \in L^2(\lambda)$. Define
	\begin{align*}
	\overline{L^2} (\lambda) = \left\{ \tilde{f} \in L^2(\lambda):\int \tilde{f} = 0   \right\},
	\end{align*}
	which is a closed subspace of $L^2(\lambda)$. One useful fact for the study of density-valued functional objects is that there exists an isometric isomorphism between $B^2(\lambda)$ and $\overline{L^2}(\lambda)$.  
	
	\begin{lemma}\label{bsplem2}
		The following hold:
		\begin{enumerate}
			\item [\upshape{(a)}] $B^2(\lambda)$ is isometrically isomorphic to $\overline{L^2} (\lambda)$.
			\item [\upshape{(b)}] $\clr : B^2(\lambda) \to \overline{L^2} (\lambda)$ defined as 
			\begin{align*}
			\clr (f)  = \log f - \int \log f\mathrm{d}\lambda, \quad f \in B^2(\lambda),
			\end{align*}
			is an isometrical isomorphism. Moreover, $\clr^{-1}$ is given by
			\begin{align*}
			\clr^{-1} (\tilde{f}) \simeq \exp(\tilde{f}), \quad \tilde{f} \in \overline{L^2} (\lambda).
			\end{align*} 
		\end{enumerate}
	\end{lemma}
	\noindent See \cite{Boogaart2014} for more details. The map $\clr$ is called the centered log-ratio ($\clr$) transformation. Since $\clr$ is a unitary linear map, the following are naturally implied:
	\begin{align*}
	&\clr(f \oplus g) = \clr(f) + \clr(g), \quad \clr(a \odot f) = a \cdot \clr (f), \\
	& \langle f,g \rangle_{B^2(\lambda)} = \langle \clr(f),\clr(g) \rangle_{L^2(\lambda)} = \int \clr(f) \clr(g)\mathrm{d}\lambda.
	\end{align*}

	Let $\lambda$ be the uniform measure defined on a compact interval $K$. The Bayes Hilbert space associated with $\lambda$, $B^2(\lambda)$,  was considered in \cite{Egozcue2006}, and it may be understood as the space of probability densities with support $K$. Note that $B^2(\lambda)$ is the space of equivalence classes of positive functions defined on $K$ with square-integrable logarithms with respect to the uniform measure $\lambda$. Within an equivalence class, the integral constraint $\int_K f \mathrm{d} \lambda = 1$ singles out the representative element as a probability density function. Then for $f,g \in B^2(\lambda)$, two operations $\oplus$ and $\odot$ are be defined as
	\[f \oplus g (x) = \frac{f(x)g(x)}{\int_K f(y)g(y)\mathrm{d}y},\quad a \odot f (x) =  \frac{f(x)^a}{\int_K f(y)^a \mathrm{d}y}. \]
	If we denote the indicator function on $K$ by $\mathbbm{1}_K$, then the $\simeq$-equivalence class associated with $\mathbbm{1}_K$ is the neutral element of perturbation. Naturally, $1 \in \mathbb{R}$ is the neutral element of powering.  Perturbation and powering are constructed such that the results of $f \oplus g$ and $a \odot f$ are always densities even if one or both of $f$ and $g$ are not proper densities. If one regards an element in $B^2(\lambda)$ as a probability density, perturbation can be interpreted as a Bayes updating. This is why $B^2(\lambda)$ is called a Bayes (Hilbert) space in the previous literature. 
	
	The fact that the resulting density from any arbitrary linear combination of two elements of $B^2(\lambda)$ is a proper density is crucial to the approach taken in this paper. In fact, the point of \cite{Beare2017}'s comments on \cite{Chang2016152} is rooted in the fact that linear combinations of probability densities in $L^2(K)$ are not necessarily probability densities. It will be shown in the subsequent sections that the Bayes Hilbert space is not just a good candidate on which to define (cointegrated) density-valued processes, but also its Hilbert space structure makes it possible to use well-developed statistical techniques in Hilbert spaces without further technicality.   
	
	We will consider a time series of densities with common support $K$. Therefore, $\lambda$ will be assumed to be the uniform measure on $K$ and density-valued linear processes will be defined in the associated Bayes Hilbert space $B^2(\lambda)$. Allowing the reference measure to be another measure with compact support is trivial. However, generalization to a reference measure with unbounded support is nontrivial and we will not consider this case. When the reference measure has unbounded support, the space of probability measures as a subset of $B^2(\lambda)$ is not a subspace, as the following example demonstrates.
	
	\begin{example} \label{bspex1}
		Let $\lambda$ be the standard normal measure defined on the real line and $B^2(\lambda)$ be the associated Bayes Hilbert space.  Let $\nu$ be the normal measure with mean $0$ and variance $2$. Then $f \coloneqq \mathrm{d}\nu/\mathrm{d}\lambda \simeq \exp(x^2/4 )$, so
		$	\int |\log f|^2 \mathrm{d} \lambda < \infty$ since the fourth moment of the standard normal measure is finite. However, note that $f$ itself represents an infinite measure; i.e.\ it is not integrable. So in this case, an element of $B^2(\lambda)$ need not be a proper density.
	\end{example}
	
	\cite{Aitchison1982} introduced the space $B^2(\lambda)$ equipped with $(\oplus, \odot)$ and $\langle \cdot , \cdot \rangle_{B^2(\lambda)} $ in the special case where $\lambda$ is the discrete uniform measure supported on a finite set $K$. The term Aitchison geometry is used to connote the geometric properties of this Bayes Hilbert space. The term generalized Aitchison geometry may be preferred when referring to Bayes Hilbert spaces constructed with general reference measures.
	
	\subsection{Bounded linear operators on Bayes Hilbert space}
	Here we briefly summarize essential concepts on bounded linear operators. The reader is referred to \cite{Bosq2000} and \cite{Conway1994} for a detailed introduction.
	
	Let $\mathcal H$ be $B^2(\lambda)$ or its complexification $B^2_{\mathbb{C}}(\lambda)$, defined as
	\begin{align*}
	B^2_{\mathbb{C}}(\lambda) = \{ f \oplus i g : f,g \in B^2(\lambda) \},
	\end{align*}
	where $i$ denotes the imaginary unit. $B^2_{\mathbb{C}}(\lambda)$ is understood to be equipped with the inner product $\langle f \oplus i g, f' \oplus i g' \rangle_{B^2_{\mathbb{C}}(\lambda) } = \langle f, f' \rangle_{B^2(\lambda) } - i \langle f, g' \rangle_{B^2(\lambda) }  + i \langle g, f' \rangle_{B^2(\lambda) } + \langle g, g' \rangle_{B^2(\lambda) }$.  
	
	Let $\mathfrak B(\mathcal H)$ be the space of bounded linear operators, equipped with the usual operator norm 
	\[ \| A \|_{\mathfrak B(\mathcal H)} = \sup_{\|f\|_{\mathcal H} \, \leq  \, 1} \| A f  \|_{\mathcal H}, \quad A \in \mathfrak B(\mathcal H).\]
	
	$A \in \mathfrak B(\mathcal H)$ is said to be compact if for two orthonormal bases $(u_j, j \in \mathbb{N})$ and $(v_j, j \in \mathbb{N})$ of $\mathcal H$, it can be written as
	\[A f =\bigoplus_{j=1}^{\infty} \gamma_j \langle f,u_j\rangle_{\mathcal H}  \odot v_j,\]
	for some sequence $(\gamma_j, j\in \mathbb{N})$ tending to zero. 
	
	Given $A \in  \mathfrak B(\mathcal H)$, we define two fundamental subspaces, the range and kernel of $A$, as follows.
	\begin{align*}
	&\ker A \coloneqq \{ f \in \mathcal H : A f =\lambda\}, \\
	&\ran A \coloneqq \{ A f : f \in \mathcal H \}.
	\end{align*} 
	The dimension of $\ker A$ is called the nullity of $A$, and the dimension of $\ran A$ is called the rank of $A$.
	
	For $A \in \mathfrak B(\mathcal H)$, there exists an operator $A^* \in \mathfrak B(\mathcal H)$, called the adjoint of $A$,  that is uniquely determined by the following property
	\begin{align*}
	\langle A f, g \rangle_{\mathcal H} = \langle f , A^* g\rangle_{\mathcal H}, \quad \text{ for all $f,g \in \mathcal H$}.
	\end{align*} 
	
	Given a subset $V \subset \mathcal H$, the orthogonal complement of $V$ is denoted by
	\begin{align*}
	V^\perp = \{g \in \mathcal H : \langle f,g \rangle_{\mathcal H} = 0 \text{ for all $f \in V$}\}.
	\end{align*}
	The closure of $V$, defined as the union of $V$ and its limit points, is denoted by $\cl(V)$. It turns out that the following relationship holds \citep[see e.g.][pp.\ 35-36]{Conway1994}
	\begin{align} \label{ranknullity}
	\ker A = (\ran A^*)^\perp \quad \text{and} \quad (\ker A)^\perp = \cl (\ran A^*).
	\end{align}
	
	$A \in  \mathfrak B(\mathcal H)$ is said to be positive semidefinite if $\langle A f, f \rangle_{\mathcal H} \geq 0$ for all $f \in \mathcal H$. If the inequality is strict for all $f\neq\lambda$, $A$ is said to be positive definite. 
	
	Given a subspace $V \in \mathcal H$, $A{\mid_V}$ denotes the restriction of an operator $A \in B(\mathcal H)$, i.e. $A{\mid_V} : V \to \mathcal H$.  
	
	Let $\mathcal K=\overline{L^2}(\lambda)$ if $\mathcal H = B^2(\lambda)$, and let $\mathcal K=\overline{L^2_{\mathbb{C}}}(\lambda)$ if $\mathcal H = B_\mathbb{C}^2(\lambda)$, where $\overline{L^2_{\mathbb{C}}}(\lambda)$ is the complexification of $\overline{L^2}(\lambda)$ similarly defined as $B_\mathbb{C}^2(\lambda)$. Let $\clr_\mathbb{C}$ denote the complexification of $\clr$, i.e. $\clr_\mathbb{C} (f\oplus i g) = \clr (f) + i \cdot \clr(g)$. Then $\clr_\mathbb{C}$ is an isometric isomorphism between $B_\mathbb{C}^2(\lambda)$ and $\overline{L^2_{\mathbb{C}}}(\lambda)$. Since $\clr$ or $\clr_\mathbb{C}$ is an isometric isomorphism between $\mathcal H$ and $\mathcal K$, any bounded linear operator in $\mathfrak B (\mathcal H)$ can be understood as the corresponding element in $\mathfrak B(\mathcal K)$. This property will be used repeatedly in subsequent sections.
	
	
	\subsection{Random elements of Bayes Hilbert space}
	Here we summarize some essential concepts relating to random elements of a Bayes Hilbert space. See \cite{Bosq2000} for a detailed discussion of random elements of Banach and Hilbert spaces.
	
	Let $(\Omega, \mathcal{F}, P)$ be the underlying probability triple and let $\mathcal H = {B^2(\lambda)}$ or $B_{\mathbb{C}}^2(\lambda)$, which is understood to be equipped with its Borel $\sigma$-field. An $\mathcal H$-valued random variable is a measurable map $X:\Omega \to \mathcal H$. Let $\Vert \cdot \Vert_{\mathcal H}$ be the norm defined on $\mathcal H$. We say that $X$ is integrable if $E \Vert X\Vert_\mathcal H < \infty$. If $X$ is integrable, there exists a unique element, denoted by $E X$, such that 
	\begin{align*}
	E \langle X , f\rangle_{\mathcal H} =  \langle E X , f\rangle_{\mathcal H}, \quad \text{ for all $f \in \mathcal H$}. 
	\end{align*}  
	$E X$ is called the expectation of $X$.
	
	If  $E \Vert X\Vert_{\mathcal H}^2 < \infty$ then $X$ is said to be square-integrable. Let $\mathfrak L^2_{\mathcal H}$ denote the space of square-integrable random variables $X$ with $E X = \lambda$, equipped with the norm 
	\begin{align*}
	\Vert X \Vert_{\mathfrak L^2_{\mathcal H}} = \left(E \Vert X\Vert^2_{\mathcal H} \right)^{1/2}.
	\end{align*}
	It turns out that $\mathfrak L^2_{\mathcal H}$ is a Banach space. For $X \in \mathfrak L^2_{\mathcal H}$, the Cauchy-Schwarz inequality implies that $X \langle f,X\rangle_{\mathcal H}$ is integrable for all $f \in \mathcal H$. Define the operator $C_X \in \mathcal B(\mathcal H)$ by 
	\begin{align*}
	C_X (f) \coloneqq E\left(X \langle f, X\rangle_{\mathcal H}\right).
	\end{align*} 
	$C_X$ is called the covariance operator of $X$.   
	\subsection{Operator pencils}
	Let $\mathcal H$ be a complex separable Hilbert space and $U$ be an open connected set in the complex plane $\mathbb{C}$. An operator-valued map $Q : U \rightarrow \mathfrak B(\mathcal H)$ is called an operator pencil. In this paper, $\mathcal H = B_\mathbb{C}^2(\lambda)\,$ or $ \overline{L^2_{\mathbb{C}}}(\lambda) $ is considered. 
	
	For notational convenience, we put $\mathcal H = \overline{L^2_{\mathbb{C}}}(\lambda)$. An operator pencil $Q(z)$ is holomorphic on an open connected set $D$ if and only if 
	\begin{equation*}
	Q^{(1)}(z_0)\coloneqq \lim_{z\rightarrow z_0} \frac{Q(z)-Q(z_0)}{z-z_0}
	\end{equation*}
	exists in the norm of $\mathfrak B(\mathcal H)$ for each $z_0 \in D$. A well known fact is that when $Q(z)$ is holomorphic at $z=z_0$, it allows the power series expansion
	\begin{align*}
	Q(z) = \sum_{k=0}^\infty Q_k (z-z_0)^k
	\end{align*}
	for some sequence $(Q_k, k \in \mathbb{N})$ in $\mathfrak B(\mathcal H)$.

	The set $\sigma(Q) \coloneqq \{z \in U : Q(z) \text{ is not invertible}\}$ is called the spectrum of $Q$, which turns out to be closed \citep[p.\ 56]{Markus2012}. 
	
	\section{Cointegrated density-valued linear processes}\label{coden}
	Hereafter the following is always assumed without explicitly stating it.
	\begin{assumption*}[R]
		$\lambda$ is the uniform measure on a compact interval $K \subset \mathbb{R}$.
	\end{assumption*}
	\noindent Under the above assumption, $B^2(\lambda)$ is the Bayes Hilbert space studied in \cite{Egozcue2006}. The assumption is not essential but convenient. In fact, extending the subsequent results of this paper to other measures supported on $K$, such as the truncated normal measure, can be easily done with trivial modifications.
	
	First we define linear processes in $B^2(\lambda)$, called Bayes linear processes. Due to the Hilbert space structure, linear processes in $B^2(\lambda)$ may be easily defined and studied according to the work of  \cite{Bosq2000,BOSQ2007879}. As seen in Section \ref{prelim}, a Bayes linear process may be understood as a linear process of probability densities with compact support $K$.  
	
	Then, we introduce I(1) processes in $B^2(\lambda)$ based on the previous work of \cite{BSS2017} who provided a notion of cointegrated linear processes in an arbitrary Hilbert space. Of course, it can be understood as an I(1) process of probability densities with support $K$. We further investigate cointegration in $B^2(\lambda)$.
	

	
	
	\subsection{Linear processes in $B^2(\lambda)$}
	Suppose that we have a time series of densities $f=(f_t, t \geq t_0)$ with compact support $K \subset \mathbb{R}$. We assume that $f_t \coloneqq \frac{d \nu_t}{d \mu}$ for each $t$, where $\nu_t$ is a probability measure for each $t$ and $\mu$  is the Lebesgue measure defined on $\mathbb{R}$.  Since $f_t$ has compact support $K$, it is absolutely continuous with respect to 
	$\lambda$. This implies that the corresponding $\lambda$-density of  the underlying probability measure $\nu_t$ is $\mu(K)^{-1} f_t$. Thus, the unit integral constraint singles out $f_t$ itself from the equivalence class containing $f_t$. The time series of $\mu$-densities $f=(f_t, t \geq t_0)$ may therefore be regarded as a stochastic process taking values in $B^2(\lambda)$.  
	
	Let $\varepsilon = (\varepsilon_t, t\in \mathbb{Z})$ be an independent and identically distributed (iid) sequence in $\mathfrak L^2_{B^2(\lambda)}$ and $(A_k, k\geq 0)$ be a sequence in $\mathfrak B (B^2(\lambda))$ satisfying $\sum_{k=0}^\infty \|A_k\|^2_{\mathfrak B (B^2(\lambda))} < \infty$. For fixed $t_0 \in \mathbb{Z} \cup \{-\infty \}$,  the sequence $f=(f_t, t \geq t_0)$ defined as
	\begin{align}\label{linearpro}
	f_t = \bigoplus_{k=0}^\infty A_k \varepsilon_{t-k}
	\end{align}
	is convergent in $\mathfrak L^2_{B^2(\lambda)}$ \citep[Lemma 7.1]{Bosq2000}. We call the sequence $f=(f_t, t \geq t_0)$ a Bayes linear process. Bayes linear processes are necessarily stationary. 
	
	
	A Bayes linear process, \eqref{linearpro},  is said to be standard if  $\sum_{k=0}^\infty \|A_k\|_{\mathfrak B (B^2(\lambda))} < \infty$. In this case, $A \coloneqq \bigoplus_{k=0}^\infty A_k$ is convergent in $\mathfrak B (B^2(\lambda))$. As in \cite{BSS2017}, we define the long-run covariance operator $\Lambda \in \mathfrak B (B^2(\lambda))$ for a standard linear process as follows
	\begin{align*}
	\Lambda \coloneqq A C_{\varepsilon_0} A^*.
	\end{align*}
	
	
	Since $\clr$ is an isometric isomorphism, it follows that the $\clr$ image of a (standard) linear process in $B^2(\lambda)$ is a (standard) linear process in  $\overline{L^2}(\lambda)$. Denote by $\tilde{g} \in \overline{L^2}(\lambda)$ the $\clr$ image of $g \in B^2(\lambda)$. Then the  $\overline{L^2}(\lambda)$-linear process paired with \eqref{linearpro} is 
	\begin{align*}
	\tilde{f}_t = \sum_{k=0}^\infty \tilde{A}_k \tilde{\varepsilon}_{t-k},
	\end{align*}
	where $\tilde{f}_t = \clr(f_t)$, $\tilde{A}_k = \clr \circ A_k \circ \clr^{-1}$ and $\tilde{\varepsilon}=(\tilde{\varepsilon_t}, t \in \mathbb{Z})$ is an iid sequence in $\overline{L^2}(\lambda)$. 
	
	\subsection{I(1) processes in $B^2(\lambda)$ and cointegration}
	\cite{Chang2016152} were the first to study I(1) processes taking values in an infinite dimensional Hilbert space, specifically the space $\overline{L^2}(\lambda)$. Subsequently, \cite{BSS2017} considered I(1) processes in an arbitrary separable complex Hilbert space. In this section we specialize the latter setting to the Bayes Hilbert space $B^2(\lambda)$. Let the time series of random densities of interest $f = (f_t, t\geq 0)$ be a sequence in $B^2(\lambda)$. Denote the time series of first differences $\Delta f = (\Delta f_t , t \geq 1 )$ with $\Delta f_t = f_t \ominus f_{t-1}$. We say that $f = (f_t, t \geq 0)$ is I(1) if $\Delta f_t$ satisfies
	\begin{align} \label{eqi1ma}
	\Delta f_t = \bigoplus_{k=0}^\infty N_k \varepsilon_{t-k}
	\end{align} 
	for all $t \geq 1$, where $\varepsilon = (\varepsilon_t, t\in \mathbb{Z})$ is an iid sequence in $\mathfrak L^2_{B^2(\lambda)}$, and $(N_k, k \geq 0)$ is a sequence in $\mathfrak B(B^2(\lambda))$ satisfying $\sum_{k=0}^\infty k \|N_k\|_{\mathfrak B(B^2(\lambda))} < \infty$ and $N \coloneqq \bigoplus_{k=0}^\infty N_k \neq 0$.
	
	From \cite{Chang2016152} and \cite{BSS2017}, it can be shown that the above I(1) sequence allows the so-called Beveridge-Nelson decomposition,
	\begin{align*} 
	f_t = (f_0 \ominus \nu_0) \oplus N \xi_t \oplus \nu_t, \quad t \geq 0
	\end{align*}  
	where $\xi_t = \bigoplus_{s=1}^t \varepsilon_s$ and $\nu_t = \bigoplus_{k=0}^\infty {N}_k \varepsilon_{t-k}$ with ${N}_k = -1 \odot (\bigoplus_{j=k+1}^\infty N_j )$. This means that $f_t$ is obtained by combining three different components: an initial condition $f_0 \ominus \nu_0$, a random walk component $N\xi_t$, and a stationary component $\nu_t$. 
	
	Given an I(1) sequence $f=(f_t, t \geq 0)$, we define the cointegrating space associated with $f$, denoted by $\mathfrak C(f,B^2(\lambda))$, to be the set
	\begin{align*}
	\mathfrak C(f, B^2(\lambda)) \coloneqq \left\{g \in B^2(\lambda) :\left(\langle f_t, g \rangle_{B^2(\lambda)} , t \geq 0 \right) \text{ is stationary for some $f_0$}\right\}.
	\end{align*}
	Moreover, we call $\mathfrak A (f,B^2(\lambda)) \coloneqq \mathfrak C(f,B^2(\lambda))^\perp$ the attractor space. These terms to indicate such subspaces are commonly used in the literature on cointegration, see e.g.\ \cite{Johansen1996}. 
	
	Define the long-run variance operator of $\Delta f$ as $\Lambda  \coloneqq N C_{\varepsilon_0} N^*$. If $C_{\varepsilon_0}$ is positive definite, it must be the case that 
	\begin{align*}
	\ker \Lambda = \ker N^*,
	\end{align*}
    and further \cite{BSS2017} showed that the cointegrating space is equal to $\ker N^*$. 
	
	The cointegrating space (or attractor space) can be identified by analyzing its $\clr$-image in $\overline{L^2}(\lambda)$, which may be convenient in practice. Suppose that the $\clr$ image  of \eqref{eqi1ma} is
	\begin{align*}
	\Delta \tilde{f}_t = \sum_{k=0}^\infty \tilde{N}_k \tilde{\varepsilon}_{t-k},
	\end{align*}
	where $\Delta \tilde{f}_t = \tilde{f}_t - \tilde{f}_{t-1}$ and $\tilde{N}_k = \clr \circ N_k \circ \clr^{-1}$. The Beveridge-Nelson decomposition is
	\begin{align*}
	\tilde{f}_t = (\tilde{f}_0-\tilde{\nu}_0) + \tilde{N} \tilde{\xi_t} + \tilde{\nu_t}, \quad t \geq 0
	\end{align*}
	where  $\tilde{N} = \clr \circ N \circ \clr^{-1}$.
	We also define the long-run variance operator of $\Delta \tilde{f}$ as $\tilde{\Lambda} = \tilde{N} C_{\tilde{\varepsilon}_0} \tilde{N}^*$, where $C_{\tilde{\varepsilon}_0}$ is the covariance operator of $\tilde{\varepsilon}_0$. The following result shows that we can fully identify the cointegrating space in $B^2(\lambda)$ associated with $f$ from the cointegrating space in $\overline{L^2}(\lambda)$ associated with $\tilde{f}$, the $\clr$ image of $f$.
	\begin{proposition}\label{clrimageprop}
		If $\,C_{\varepsilon_0}$ is positive definite, then $\mathfrak C(f, B^2(\lambda)) = \clr^{-1} \ker \tilde{N}^*$ 
	\end{proposition} 
	\begin{proof}
		Proposition 3.1 in \cite{BSS2017} implies that $\mathfrak C(f, B^2(\lambda)) = \ker N^* $. Note that $\clr^* = \clr^{-1}$ since for $\tilde{g} \coloneqq \clr (g)$,
		\begin{align*}
		\langle f, \clr^{-1} \tilde{g} \rangle_{B^2(\lambda)} = \langle \clr (f), \tilde{g} \rangle_{\overline{L^2}(\lambda)}. 
		\end{align*}
		This holds for any arbitrary $g \in B^2(\lambda)$. By the uniqueness of the adjoint map, $\clr^* = \clr^{-1}$. From these results, it is clear that
		\begin{equation*}
	\tilde{N}^* = \clr \circ N^* \circ \clr^{-1}.
		\end{equation*}
		Since $\clr$ and $\clr^{-1}$ are bijective, it follows that $\ker N^* = \clr^{-1} \ker \tilde{N}^*$. 
	\end{proof}
	
	\cite{Chang2016152} considered a cointegrated density-valued process with values in $L^2(\lambda)$. However,  	\cite{Beare2017} pointed out that the attractor space under their assumptions is always trivial, i.e. $\mathfrak A(f, L^2(\lambda)) = \{0\}$. This problem occurs because an I(1) process of probability densities in $L^2(\lambda)$ cannot satisfy the nonnegativity constraints required for each realization to be a proper density; the space of probability densities with compact support $K$ is strictly smaller than $L^2(\lambda)$ and  not a subspace. Due to the geometry of the Bayes Hilbert space, a cointegrated linear process in $B^2(\lambda)$ does not suffer from this problem. Furthermore, it is clear from the earlier discussion of Bayes Hilbert spaces in Section \ref{bayessp} that the space of probability densities supported on $K$ is isomorphic to $\overline{L^2}(\lambda)$, which is just a slightly smaller subspace than ${L^2}(\lambda)$ itself.

	\section{Cointegrated AR($p$) processes in $B^2(\lambda)$} \label{arden}
	In this section, we study AR($p$) processes with values in $B^2(\lambda)$ and obtain a version of the Granger-Johansen representation theorem based on the analytic operator pencil theory.   
	
	Suppose that a sequence of $\lambda$-densities $(f_t, t \geq -p+1)$ in $B^2(\lambda)$ evolves according to
	\begin{align} \label{arlaw}
	f_t = A_1 f_{t-1} \oplus \cdots \oplus A_p f_{t-p} \oplus \varepsilon_t
	\end{align} 
	for $t\geq1$, where $\varepsilon = (\varepsilon_t , t \in \mathbb{Z})$ is an iid sequence in $\mathfrak L^2_{B^2(\lambda)}$ and $A_1,\ldots,A_p \in \mathfrak B(B^2(\lambda))$. 
	Equation \eqref{arlaw} may be understood as an autoregressive law of motion for density-valued processes in $B^2(\lambda)$. The corresponding $\clr$ image of  \eqref{arlaw} with values in $\overline{L^2}(\lambda)$ is given by
	\begin{align} \label{clrarlaw}
	\tilde{f_t} = \tilde{A}_1 \tilde{f}_{t-1} + \cdots + \tilde{A}_p \tilde{f}_{t-p} + \tilde{\varepsilon_t}
	\end{align}  
	for $t\geq1$, where $\tilde{A}_k = \clr \circ A_k \circ \clr^{-1} $. It is clear that $\tilde{A}_k\in\mathfrak{B}(\overline{L^2}(\lambda))$, and that $\tilde{A}_k$ is compact if and only if $A_k$ is compact.  Let $\id_{\overline{L^2}(\lambda)}$ be the identity operator on $\overline{L^2}(\lambda)$, and  for $z \in \mathbb{C}$ define the operator pencil
	\begin{align*}
	\tilde{A}(z) = \id_{\overline{L^2}(\lambda)} - z \tilde{A}_1  - \cdots - z^p \tilde{A}_p.
	\end{align*}
	For each $z\in\mathbb C$, $\tilde{A}(z)$ is a bounded linear operator on $\overline{L^2_\mathbb{C}}(\lambda)$, the complexification of $\overline{L^2}(\lambda)$. The operator pencil $\tilde{A}(z) $ is polynomial in $z$; see \cite{Markus2012} for a detailed discussion of polynomial operator pencils. Let $\id_{B^2(\lambda)}$ be the identity operator on $B^2(\lambda)$, and define 
	\begin{align}
	A(z) = \id_{B^2(\lambda)} \ominus (z \odot A_1) \ominus \cdots \ominus (z^p \odot A_p) . 
	\end{align}
	For each $z$, $A(z)$ is a bounded linear operator on $B_\mathbb{C}^2(\lambda)$.	One can easily show that $A(z) = \clr_\mathbb{C}^{-1} \circ \tilde{A}(z) \circ \clr_\mathbb{C}$, so $\sigma(A) = \sigma(\tilde{A})$ holds.
	
	To further analyze \eqref{arlaw}, we employ the following mutually exclusive assumptions.  
	\begin{assumption*}[\textbf{S}]
		If $z \in \sigma(A)$ then $|z| > 1$.
	\end{assumption*}
	\noindent Let $D_{r}$ be the open disk centered at zero with radius $r>0$. Since $\sigma(A)$ is closed, Assumption (S) implies that there exists $\eta > 0$ such that $A(z)$ is invertible on $D_{1+\eta}$. 
	
	\begin{assumption*}[\textbf{N}]\leavevmode
		\begin{itemize}
			\item[(i)]	$A_1,\ldots,A_p \in \mathfrak B (B^2(\lambda))$ are compact operators. 
			\item[(ii)]  If $z\in\sigma(A)$ then $|z| > 1$ or $z=1$. Moreover, $	1 \in \sigma(A)$.
			\item[(iii)] If $\mathbf{P}\in\mathfrak B (B^2(\lambda))$ is a projection on  $\ran A(1)$, then the map
			\[(\id_{B^2(\lambda)} \ominus \mathbf{P}) A^{(1)}(1){\mid_{\ker A(1)}} : \ker A(1) \to \ran (\id_{B^2(\lambda)} \ominus \mathbf{P})\]
			is invertible. 
		\end{itemize}		
	\end{assumption*}
	\noindent Under Assumption (N)-(i), $A(z)$ is an index-zero Fredholm operator for each $z \in \mathbb{C}$. This implies that $\ran A(1)$ is closed. (N)-(i) is not restrictive in practice since it is in fact common to assume the compactness of autoregressive operators. (N)-(ii) implies the existence of $\eta>0$ such that $A(z)$ is invertible everywhere on $D_{1+\eta}$ except at $z=1$. The role of (N)-(iii) will be discussed later. 
	
	The AR($p$) law of motion \eqref{arlaw} behaves very differently depending on whether $A(z)$ satisfies Assumption (S) or (N). In the former case it generates a stationary process which may be represented as a standard Bayes linear process.   
	\begin{proposition}\label{propstat}
		Suppose that Assumption $(\mathrm{S})$ is satisfied. Then the AR$(p)$ law of motion \eqref{arlaw} admits a unique stationary solution in $\mathfrak L^2_{B^2(\lambda)}$ given by
		\begin{equation}\label{constat2}
		f_t =\bigoplus_{k=0}^\infty N_k \varepsilon_{t-k}
		\end{equation} 
		for all $t \geq 1$. Moreover,  $N_k$ is given by
		\begin{align*}
		N_k = \frac{1}{k !} \odot N^{(k)}(0)
		\end{align*}
		and $N(z) \coloneqq A(z)^{-1}$ is holomorphic on $D_{1+\eta}$ for some $\eta > 0$. 
	\end{proposition}
	\begin{proof}[Proof of Proposition \ref{propstat}]
		Consider the $\clr$ image of \eqref{clrarlaw} in $\overline{L^2}(\lambda)$ and let $\mathcal H = \overline{L^2}(\lambda)$. Using the Markovian representation in \citet[p. 128]{Bosq2000}, the AR($p$) law of motion in $\mathcal H$ can be re-expressed as an AR(1) law of motion in $\mathcal H^p$ by writing
		\begin{align} \label{comar1}
		\mathcal F_t = \mathcal A \mathcal F_{t-1} + \mathcal E_t,
		\end{align}
		where
		\begin{align*}
		\mathcal F_t =\left[\begin{matrix}\tilde{f}_t\\ \tilde{f}_{t-1} \\ \vdots \\ \tilde{f}_{t-p+1}\end{matrix} \right], \quad  \mathcal E_t =\left[\begin{matrix}\tilde{\varepsilon_t}\\0 \\ \vdots \\ 0\end{matrix} \right], \quad \mathcal A  = \left[\begin{matrix} \tilde{A}_1 & \tilde{A}_2 & \cdots & \tilde{A}_{p-1} & \tilde{A}_p \\ \id_{\mathcal H} & 0 & \cdots &0 & 0 \\ \vdots & \vdots& \ddots & \vdots &\vdots \\ 0 &0 & \cdots &  \id_{\mathcal H} & 0\end{matrix} \right].
		\end{align*} 
		Define $\mathbb A(z) \coloneqq  \id_{\mathcal H^p} - \mathcal A z$, i.e.
		\begin{align*}
		\mathbb A(z) &= \left(\begin{array}{c:ccccc} \id_{\mathcal H} - \tilde{A}_1z & -\tilde{A}_2 z & -\tilde{A}_3z & -\tilde{A}_4z & \cdots &-\tilde{A}_pz  \\ \hdashline -\id_{\mathcal H} z & \id_{\mathcal H} & \zero  & \zero  & \cdots & \zero \\
		\zero & -\id_{\mathcal H} z & \id_{\mathcal H}  & \zero  & \cdots & \zero \\
		\vdots& \vdots & \vdots& \vdots & \ddots & \vdots  \\
		\zero & \zero  & \zero  & \zero  & -\id_{\mathcal H} z & \id_{\mathcal H} 
		\end{array}\right)\\& \eqqcolon \left(\begin{matrix} \mathbb A_{11}(z)& \mathbb A_{12}(z) \\ \mathbb A_{21}(z) & \mathbb A_{22}(z)\end{matrix}\right),
		\end{align*} 
		where one can easily verify that $\mathbb A_{22}(z)$ is invertible. Define the Schur complement of $\mathbb A_{22}(z)$ as $\mathbb A^+_{11}(z) \coloneqq \mathbb A_{11}(z) - \mathbb A_{12}(z) \mathbb A_{22}^{-1} \mathbb A_{21}(z)$. From a little algebra, it can be easily verified that $\mathbb A^+_{11}(z) = \tilde{A}(z)$. Since  $\mathbb A_{22}(z)$ is invertible, $\mathbb A(z)$ is invertible if and only if $\mathbb A^+_{11}$ is invertible. Therefore, it follows that $\sigma(\mathbb A) = \sigma(\tilde{A}) \subset \{z : |z| > 1 \}$. Define
		\[s(\mathcal A) \coloneqq \{ z \in \mathbb{C} : \mathcal A - z \id_{\mathcal H^p} \text{ is invertible} \}\]
		and
		\[\rho(\mathcal A) \coloneqq \sup \{ |z| : z \in \sigma(\mathcal A)  \},\]
		the spectrum and spectral radius of the bounded linear operator $\mathcal A$. From the construction of the operator pencil $\mathbb A(z)$ and Assumption (S), it is clear that 
		\begin{align} 
		\rho(\mathcal A) \coloneqq \sup \{ |z| : z \in s(\mathcal A) \} < 1. \label{spectradius}
		\end{align} 
		In view of Gelfand's formula \citep[Proposition VII.3.8]{Conway1994}, \eqref{spectradius} implies that there exists $j$ such that 
		\begin{align} \label{radiusconv}
		\|\mathcal A^j\|_{\mathfrak B(\mathcal{H}^p)} < a^j, \quad a\in (0,1).
		\end{align}
		From a well known result on the inverse of Banach-valued functions
		, $\mathbb{A}(z)^{-1}$ is given as
		\begin{align*}
		\mathbb{A}(z)^{-1} = \id_{H} + \mathcal Az + \mathcal A^2 z^2 + \cdots 
		\end{align*}
		which is convergent in $D_{1+\eta}$, and indeed holomorphic on $D_{1+\eta}$, for some $\eta > 0$ because of  \eqref{radiusconv}. This shows that  the AR(1) law of motion in \eqref{comar1} may be represented as
		\begin{align*}
		\mathcal F_t = \sum_{k=0}^\infty \mathcal A^k  \mathcal E_{t-k}, \quad t \in \mathbb{Z}.
		\end{align*}
		Let $\Pi : \mathcal H^p \to \mathcal H$ be the coordinate projection given as $\Pi(x_1,\ldots,x_p) = x_1$ and let $\Pi^*: \mathcal H \to \mathcal H^p$ be its adjoint. Then  one can easily verify that
		\begin{align*}
		\tilde{f}_t = \sum_{k=0}^\infty  \Pi \mathcal A^k \Pi^*\,   \tilde{\varepsilon}_{t-k}, \quad t \in \mathbb{Z}.
		\end{align*}
		A simple algebra yields $\Pi \mathbb{A}(z)^{-1} \Pi^*  = \mathbb{A}_{11}^+(z)^{-1} = \tilde{A}(z)^{-1}$, and it is clear that $\Pi \mathcal A^k \Pi^*$ is the $k$th coefficient in the Taylor series of $\tilde{A}(z)^{-1}$ around $z=0$.
		
		Holomorphicity of $A(z)^{-1}$ on $D_{1+\eta}$ is inherited from the holomorphicity of $\tilde{A}(z)^{-1}$ on $D_{1+\eta}$, which is implied by the holomorphicity of $\mathbb{A}(z)^{-1}$ established above and the relation $\tilde{A}(z)^{-1}=\Pi \mathbb{A}(z)^{-1} \Pi^*$.
	\end{proof}
	
	\begin{remark}
		Norm-summability, $\sum_{k=0}^\infty \|\tilde{N}_k\|_{\mathfrak B(B^2(\lambda))} < \infty$, is a natural consequence of holomorphicity of $A(z)^{-1}$ on $D_{1+\eta}$ for some $\eta>0$. This shows that  \eqref{constat2} is a standard Bayes linear process.
	\end{remark}
	
	Under Assumption (N), a process in $B^2(\lambda)$ satisfying the AR($p$) law of motion \eqref{arlaw} is no longer stationary. Loosely speaking, it corresponds to an I($d$) process for $d \geq 1$.

	\begin{proposition} \label{propi1}
		Under Assumption $(\mathrm{N})$-$(\mathrm{i},\mathrm{ii})$, the AR$(p)$ law of motion does not admit a stationary solution. Instead, it allows the following I($d$) solution.
		\begin{equation*}
		f_t = \tau \oplus  {N}_{-d} {\xi}_{(d),t} \oplus  {N}_{-d+1} {\xi}_{(d-1),t} \oplus \cdots \oplus  {N}_{-1} {\xi}_{(1),t} + {\nu}_t, 
		\end{equation*} 
		where  $\xi_{(0),t} = \varepsilon_t$, $\xi_{(\ell),t} = \bigoplus_{s=1}^t \xi_{(\ell-1),s}$ for $\ell=1,\ldots,d$, $\tau = \bigoplus_{j=0}^{d-1} t^j \tau_j$, $(\tau_0,\ldots,\tau_{d-1})$ are time invariant elements of $\mathfrak L^2_{B^2(\lambda)}$,  $(N_{j}, j=-d,\ldots,-1)$ are the coefficients of $(1-z)^j$ in the Laurent expansion of $A(z)^{-1}$. 
		Moreover, $\nu_t = \bigoplus_{j=0}^\infty {N}^H_{j,0} {\varepsilon}_{t-j}$ and ${N}^H_{j,0}$ is the coefficient of $z^j$ in the Taylor expansion of ${N}^H(z)$, the holomorphic part of the Laurent expansion of $A(z)^{-1}$,  around $z=0$. 	
		
		Assumption $(\mathrm{N})$-$(\mathrm{iii})$ is necessary and sufficient for $d=1$, and we have 
		\begin{align}\label{eqresidue}
		\lim_{z \to 1} (1-z) \odot A(z)^{-1} = - \left[\left(\id_{B^2(\lambda)} \ominus \mathbf{P}\right) A^{(1)}(1){\mid_{\ker A(1)}}\right]^{-1} \left(\id_{B^2(\lambda)} \ominus  \mathbf{P}\right),
		\end{align}
		which does not depend on the choice of the projection $\mathbf{P}$ on $\ran A(1)$.
	\end{proposition}
	\begin{proof}
		For notational convenience, we work with $\tilde{A}(z)$ and set $\mathcal{H} = \overline{L^2}(\lambda)$. Assumption (N)-(i,ii) implies that $1$ is an isolated singularity of $\tilde{A}(z)$.    
		From Lemma \ref{lem1}, we know that there exist finite dimensional projections $Q_1,\ldots,Q_d$ such that, for $z$ in a punctured neighborhood of 1,
		\begin{align*}
		\tilde{A}(z) = (P_1 + (1-z)Q_1) \cdots(P_d + (1-z)Q_d)G(z),
		\end{align*}	
		where $P_i = \id_\mathcal{H} - Q_i$ for $i=1,\ldots,d$, and $G(z)$ is a holomorphic function that is invertible at $z=1$. 
		Note that for $i=1,\ldots,d$
		\begin{align*}
		(P_i + (1-z) Q_i)^{-1} = (P_i + (1-z)^{-1} Q_i).
		\end{align*}
		This shows that in a punctured neighborhood of $z=1$,
		\begin{align}\label{howlandexpansion}
		\tilde{A}(z)^{-1} &= G(z)^{-1} (P_d + (1-z)^{-1} Q_d) \cdots (P_1 + (1-z)^{-1}Q_1).
		\end{align}	
		Therefore, $\tilde{A}(z)^{-1}$ is meromorphic and the lowest order of the Laurent series is $-d$. Moreover, 	$\tilde{N}(z)\coloneqq(1-z)^d \tilde{A}(z)^{-1}$  is given as follows. 
		\begin{align} \label{eqadd01}
		\tilde{N}(z) = \tilde{N}_{-d} + \tilde{N}_{-d+1}(1-z) + \cdots + \tilde{N}_{-1}(1-z)^{d-1} +  \tilde{N}^H(z) (1-z)^d,
		\end{align} By applying the linear filter induced by \(\tilde{N}(z)\) to \eqref{arlaw}, we obtain the following difference equation,
		\begin{align}\label{eqadd02}
		\Delta^d \tilde{f}_t = \tilde{N}_{-d} \tilde{\varepsilon}_t + \tilde{N}_{-d+1}(\Delta \tilde{\varepsilon}_t) + \cdots + \tilde{N}_{-1}(\Delta^{d-1} \tilde{\varepsilon}_t) + \Delta^d\nu_t,
		\end{align} where $\nu_t = \sum_{j=0}^\infty \tilde{N}^H_{j,0} \tilde{\varepsilon}_{t-j}$ and $\tilde{N}^H_{j,0} = \mathrm{d}^j\tilde{N}^H(z)/\mathrm{d}z^j \mid_{z=0}$.
		Clearly, the process given by 
		\begin{align}
		\tilde{f}_0^* = \nu_0, \quad \tilde{f}_t^* = \tilde{N}_{-d} \tilde{\xi}_{(d),t} +  \tilde{N}_{-d+1} \tilde{\xi}_{(d-1),t} + \cdots +  \tilde{N}_{-1} \tilde{\xi}_{(1),t} + \tilde{\nu}_t 
		\end{align}
		is a solution to the difference equation \eqref{eqadd02}. It is completed by adding the solution to $\Delta \tilde{f}_t = 0$, which is given by $\sum_{j=0}^{d-1} t^{j} \tilde{\tau}_j$ for some time invariant $\tau_0,\ldots,\tau_{d-1} \in \mathfrak L^2_{B^2(\lambda)}$.  
		This proves the first statement. 
		
		Additionally under Assumption (N)-(iii), Lemma \ref{lem2} implies that $d=1$. The only remaining thing is to verify the residue formula \eqref{eqresidue}. From \eqref{howlandexpansion} we know that
		\begin{align*}
		\lim_{z\to 1}(1-z) \tilde{A}(z)^{-1} &= G(1)^{-1} (\id_\mathcal{H}- P_1).
		\end{align*}	
		Define
		\begin{align*}
		&\mathbb{A}\coloneqq G(1)^{-1}{\mid_{\ran(\id_{\mathcal H}-P_1)}},\\
		&\mathbb{B} \coloneqq G(1){\mid_{\ker \tilde{A}(1)}}. 
		\end{align*}
		From Lemma \ref{lem3}(a,b) we know that $\mathbb{B}=  - (\id_\mathcal{H}-P_1) \tilde{A}^{(1)}(1){\mid_{\ker \tilde{A}(1)}}$ and $\ran \mathbb{B} \subset \ran(\id_{\mathcal H}-P_1)$. Since $\mathbb{B}$ as a map from $\ker \tilde{A}(1)$ to $\ran(\id_{\mathcal H}-P_1)$  is invertible under Assumption (N)-(iii), it is clear that $\mathbb{A}\mathbb{B} = \id_{\ker \tilde{A}(1)}$. This shows that 
		\begin{align*}
		\tilde{N}_{-1} \coloneqq \lim_{z\to 1}(1-z) \tilde{A}(z)^{-1}=\mathbb A(\id_{\mathcal H}-P_1) = \mathbb{B}^{-1} (\id_\mathcal{H}-P_1).
		\end{align*}
		Under Lemma \ref{lem1}, we may choose $P_1$ to be any projection on $\ran \tilde{A}(1)$, so the expression for $\tilde{N}_{-1}$ just obtained cannot depend on the particular choice of projection.
	\end{proof}
	\noindent 	Under Assumption (N)-(i,ii), Proposition \ref{propi1} says that the AR($p$) law of motion \eqref{arlaw} allows I(d) solutions for $d\geq 1$. The most common case in practice may be I(1), and such an I(1) solution is guaranteed under an extra condition, Assumption (N)-(iii).

	\begin{remark}
		When $\mathcal H = \mathbb{R}^n$, \cite{Johansen1991} provided a necessary and sufficient condition on the autoregressive matrix polynomial and its first derivative at one for an AR($p$) law of motion to be I(1). It is commonly called the I(1) condition. Assumption (N)-(iii) plays the same role for an AR($p$) law of motion in $B^2(\lambda)$. \cite{BSS2017} provided a sufficient condition for the existence of I(1) solution when $p>1$ in an arbitrary complex Hilbert space, and Assumption (N)-(iii) is in fact a reformulation of it. However, note that we showed that (N)-(iii) is a necessary and sufficient condition for the existence of I(1) solution. Therefore, Proposition \ref{propi1} extends the version of the Granger-Johansen representation theorem in \cite{BSS2017} when $p > 1$. 
	\end{remark}
	
	\begin{remark}
		When $p=1$, an I(1) solution can be guaranteed without requiring compactness of the autoregressive operator. See Theorem 4.1 in \cite{BSS2017}.
	\end{remark}
	
	Under Assumption (N)-(i,ii,iii), one natural consequence of Proposition \ref{propi1} is the following Beveridge-Nelson decomposition: 	for some $f_0, \nu_0 \in B^2(\lambda)$,	\begin{align}\label{bndecomp}
	f_t = \left(f_0 \ominus \nu_0\right) \oplus N(1) \xi_t \oplus \nu_t, \quad t \geq 0
	\end{align}
	where $\xi_t = \bigoplus_{s=1}^t \varepsilon_s$, $\nu_t = \bigoplus_{k=0}^\infty N_k \varepsilon_{t-k}$ and $N(1)$ is a finite rank operator which can be explicitly obtained from the residue formula \eqref{eqresidue}. If $\varepsilon = (\varepsilon_t, t \in \mathbb{Z})$ has a positive definite covariance operator,  one can easily identify the cointegrating and attractor spaces of $B^2(\lambda)$ from the decomposition \eqref{bndecomp}.
	\begin{proposition}
		Assume that $C_{\varepsilon_0}$ is positive definite and  $(\mathrm{N})$-$(\mathrm{i},\mathrm{ii},\mathrm{iii})$ hold. Then $\mathfrak C (f,B^2(\lambda)) = [\ker A(1)]^\perp$. 
	\end{proposition}
	\begin{proof}
		We only provide an informal argument. For a more detailed and rigorous proof, refer to \citet[Proposition 3.1]{BSS2017}. From \eqref{bndecomp}, for any $g\in B^2(\lambda)$ we have
		\begin{align*}
		\langle f_t, g\rangle_{B^2(\lambda)}  = \langle f_0 \ominus \nu_0 , g\rangle_{B^2(\lambda)} + \langle N(1) \xi_t , g\rangle_{B^2(\lambda)} + \langle \nu_t, g\rangle_{B^2(\lambda)}.
		\end{align*}   
		By employing the initial condition $f_0 = \nu_0$, we can make the first term vanish. Moreover, the inner product process $(\langle \nu_t, g \rangle_{B^2(\lambda)}, t \geq 0)$ is stationary because $(\nu_t,t\geq0)$ is stationary. The variance of $\langle  N(1) \xi_t, g \rangle_{B^2(\lambda)}$ increases in $t$. Thus, $\langle  N(1) \xi_t, g \rangle_{B^2(\lambda)}$ must vanish in order for $\langle f_t, g\rangle_{B^2(\lambda)}$ to be stationary. Given that $\varepsilon$ has a positive definite covariance operator, for $\langle  N(1) \xi_t, g \rangle_{B^2(\lambda)}$ to vanish we require $g \in \ker  N(1)^*$. From the closed form solution for $ N(1)$ given in \eqref{eqresidue} combined with \eqref{ranknullity} , it can be easily deduced that $\ker  N(1)^* = [\ker A(1)]^\perp$. 
	\end{proof}
	\noindent Naturally, the attractor space $\mathfrak A (f,B^2(\lambda)) = \mathfrak C (f,B^2(\lambda))^\perp = \ker A(1)$, which is finite dimensional under Assumption (N). 
	
	
	\begin{example}[A numerical example] \hspace{0.1cm}
		
		\noindent Let $\lambda$ be the uniform measure on $[-3,3]$ and consider the Bayes Hilbert space associated with $\lambda$. Let $g$ be the standard normal measure truncated on $[-3,3]$. We consider the following AR(1) law of motion
		\begin{align*}
		&(f_t \ominus g) = \Psi (f_{t-1} \ominus g) \oplus \varepsilon_t, 
		\\ & \Psi(\cdot) = \bigoplus_{j=1}^\infty \lambda_j \langle \cdot , e_j \rangle_{B^2(\lambda)} \odot e_j,
		\end{align*}
		where $\lambda_j = 2^{1-j}$. Let $(u_j, j \in \mathbb{N})$ be the Fourier basis for $\overline{L^2}(\lambda)$, and let $e_1$ be the Cauchy probability density with location zero and scale 0.25. We obtain an orthonormal basis $(e_j, j \in \mathbb{N})$ for $B^2(\lambda)$ by applying the Gram-Schmidt process to the basis $( e_1,\clr^{-1}u_2, \clr^{-1}u_3, \ldots ) $. If $\varepsilon_t$ has a positive definite covariance operator, this specification makes the attractor space equal to 
		\begin{align} \label{nuexatt}
		\mathfrak A(f,B^2(\lambda)) = \{ g \oplus (a \odot e_1) : a \in \mathbb{R} \}.
		\end{align} Figure \ref{fig1} shows typical elements in the attractor space. An element $g \oplus (a \odot e_1)$ shows higher concentration around $0$ when $a > 0$, or is bimodal when $a < 0$. 
		To obtain a simulated sequence, I generated a sequence of independent standard Brownian bridges $(B_t,t\geq1)$, and set
		\begin{align*}
		\varepsilon_t  = \clr^{-1} \{0.3 \cdot \, \mathbf{P}_m(\lambda) \, \overline{B}_t \},
		\end{align*}
		where $ \mathbf{P}_m(\lambda)$ denotes orthogonal projection on the space of $m$th-order polynomials in $\overline{L^2}(\lambda)$, and $\overline{B}_t(u) = B_t((u+3)/6) - 6 \int_{0}^1 B_t(r) \mathrm{d}r$ for $u\in[-3,3]$. A consequence of using the projection $\mathbf P_m$ to construct the innovation $\varepsilon_t$ is that the attractor space may not be exactly equal to that in \eqref{nuexatt}, since the covariance operator of $\varepsilon_t$ is not positive definite. However, this is a convenient way to force the simulated probability densities to be smooth, and by choosing large $m$ the attractor space should approximate that in \eqref{nuexatt}. I set $f_1 = g$ and generated a sequence which is shown in Figure \ref{fig2}. We can easily see that $f_t$ floats around the attractor space in the long run.
	
			\begin{figure}[h!] 
				\centering
			\caption{Typical elements in the attractor} 	\label{fig1} 
			\begin{tabular}{cc}
				\scriptsize{\quad$g \oplus e_1$} & \scriptsize{\quad$g \ominus e_1$}\\
				\includegraphics[width=0.4\textwidth]{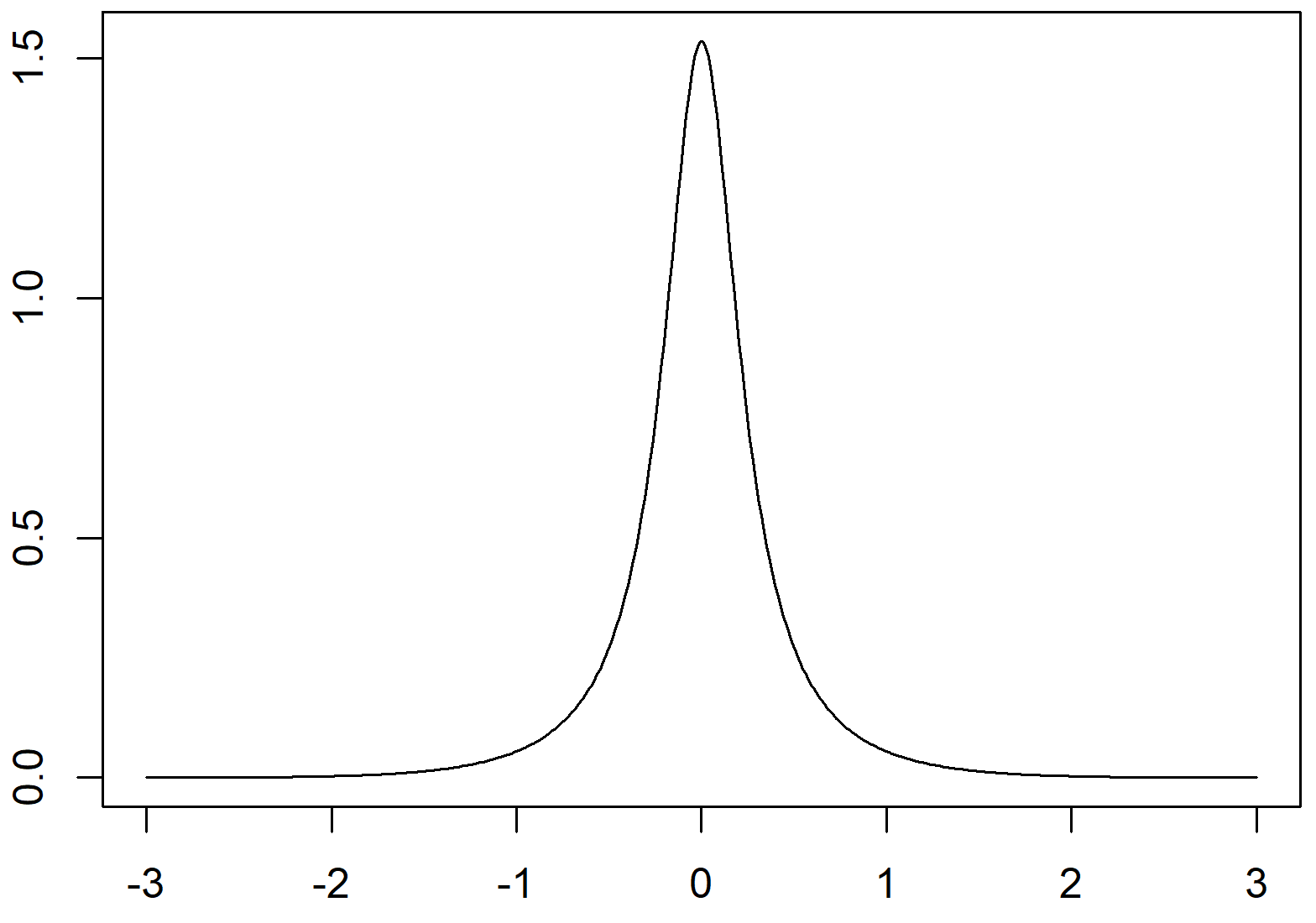} & 	\includegraphics[width=0.4\textwidth]{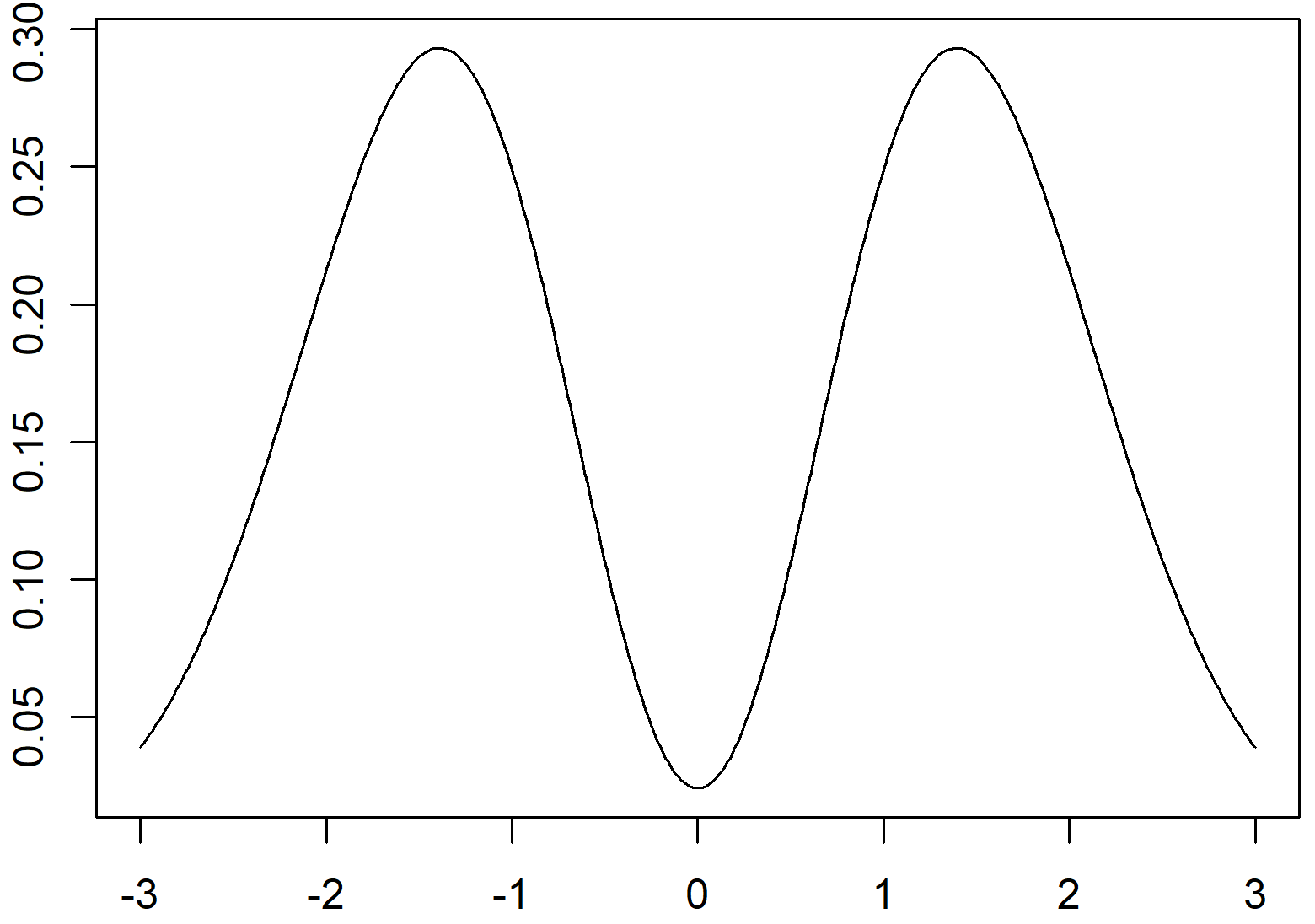}	
			\end{tabular}
		\end{figure}
		\begin{figure}[h!]
			\centering 
			\caption{Realization of $(f_t, t\geq 1)$}	\label{fig2}
			\begin{tabular}{cccc}
				\scriptsize{$t=1$} & \scriptsize{$t=10$}& \scriptsize{$t=30$}& \scriptsize{$t=100$}\\
			\includegraphics[width=0.22\textwidth]{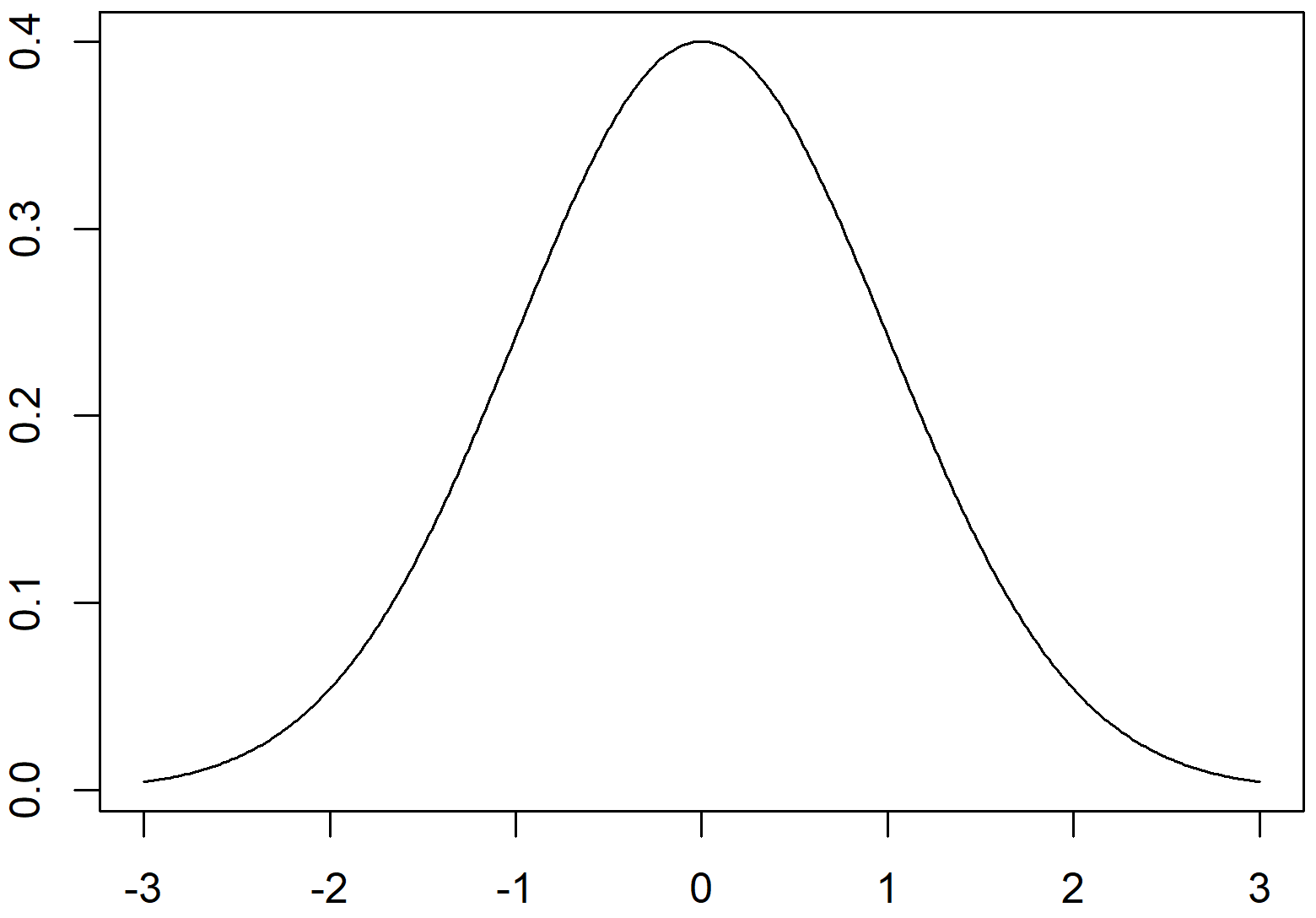}&	\includegraphics[width=0.22\textwidth]{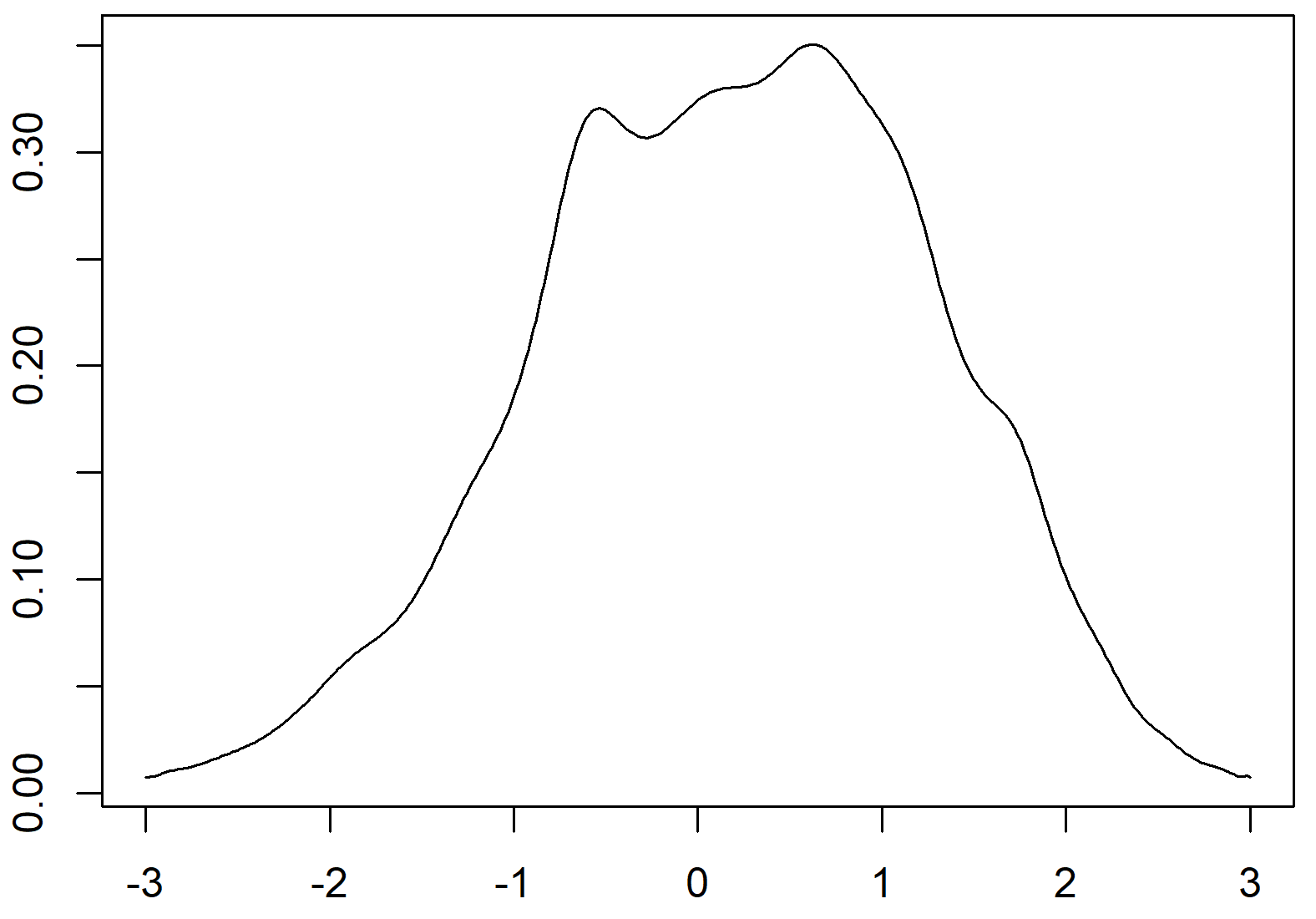}&\includegraphics[width=0.22\textwidth]{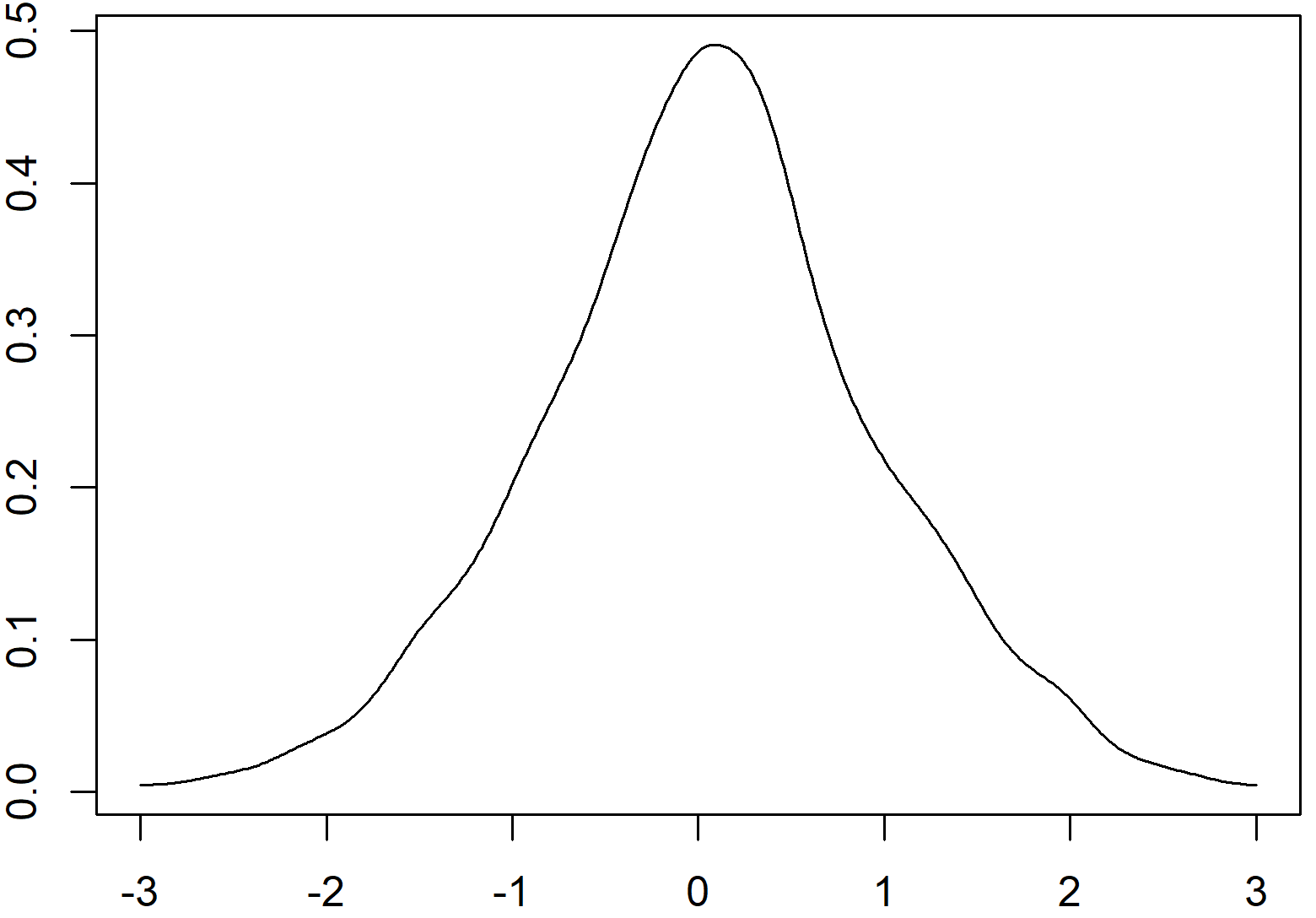}&\includegraphics[width=0.22\textwidth]{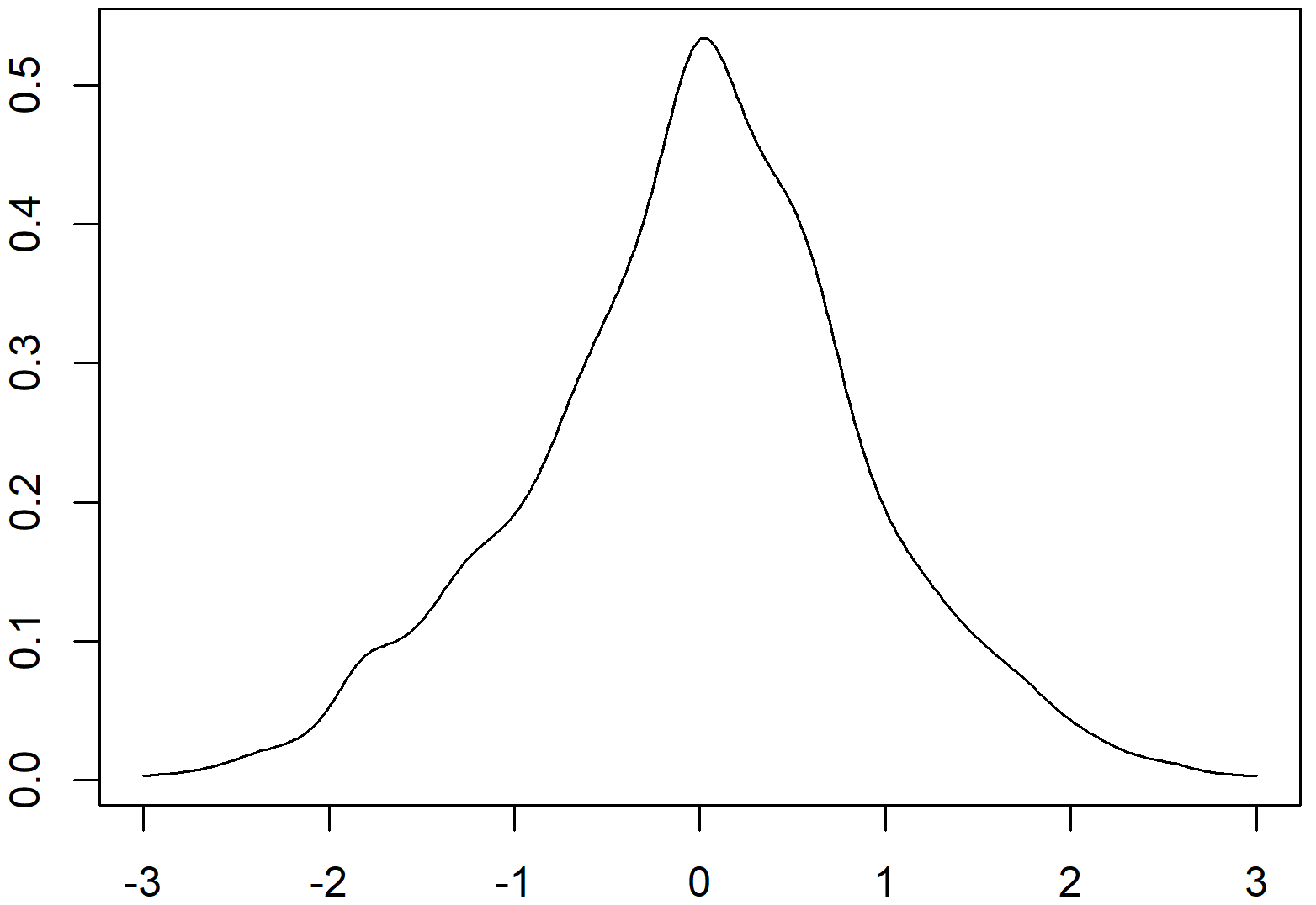}\\	
			\scriptsize{$t=300$} & \scriptsize{$t=500$}& \scriptsize{$t=800$}& \scriptsize{$t=1000$}\\
		\includegraphics[width=0.22\textwidth]{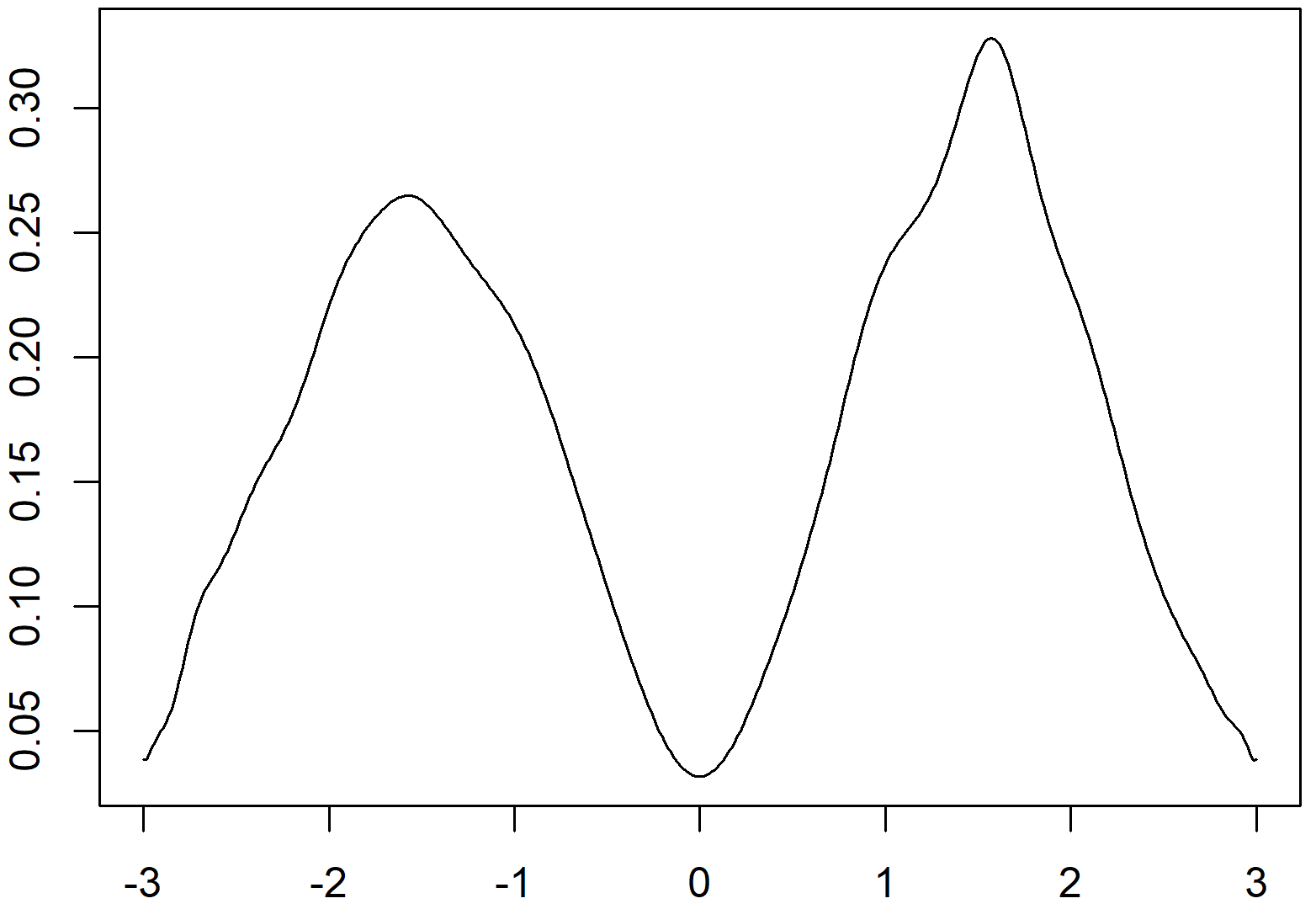}&	\includegraphics[width=0.22\textwidth]{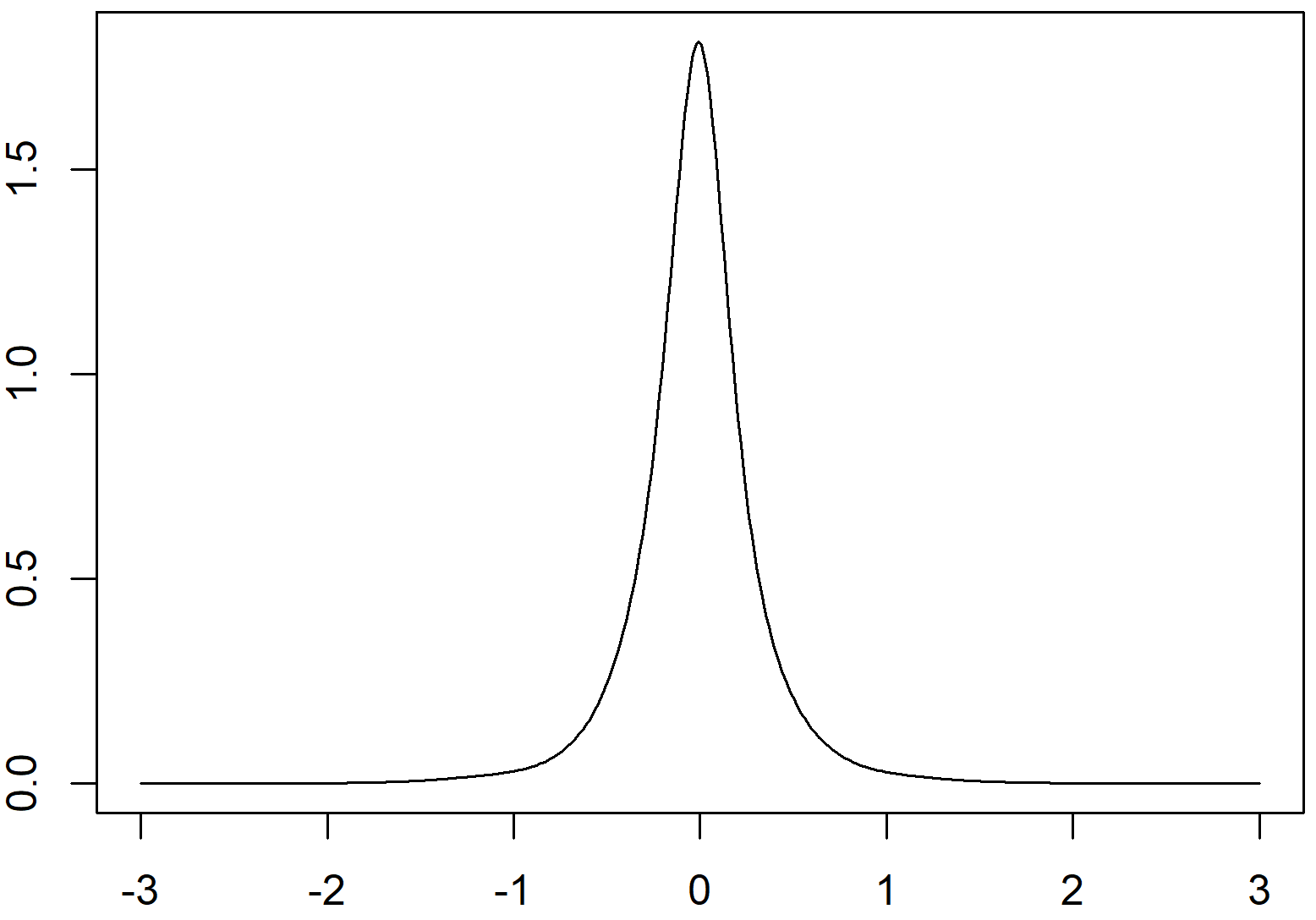}&\includegraphics[width=0.22\textwidth]{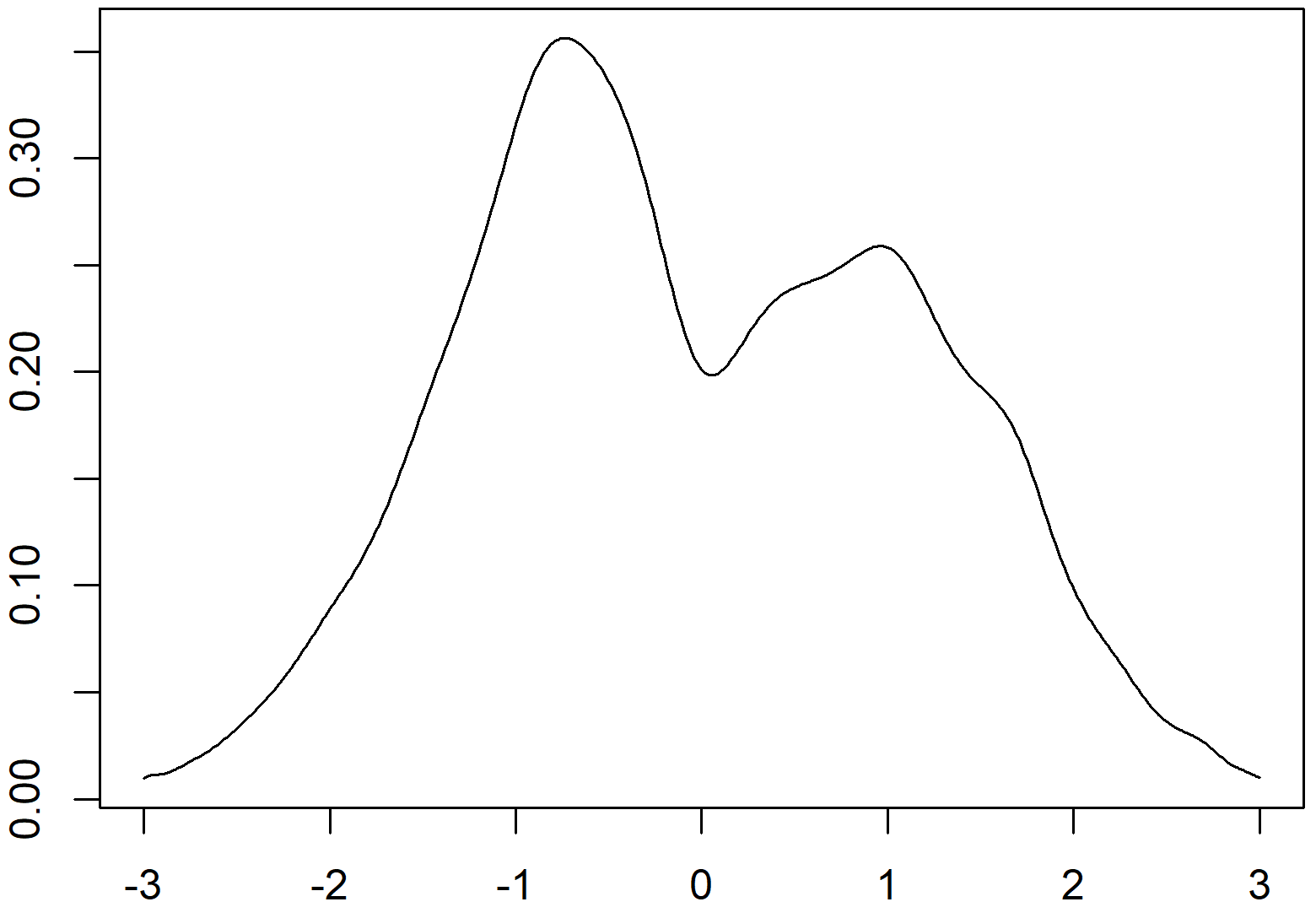}&\includegraphics[width=0.22\textwidth]{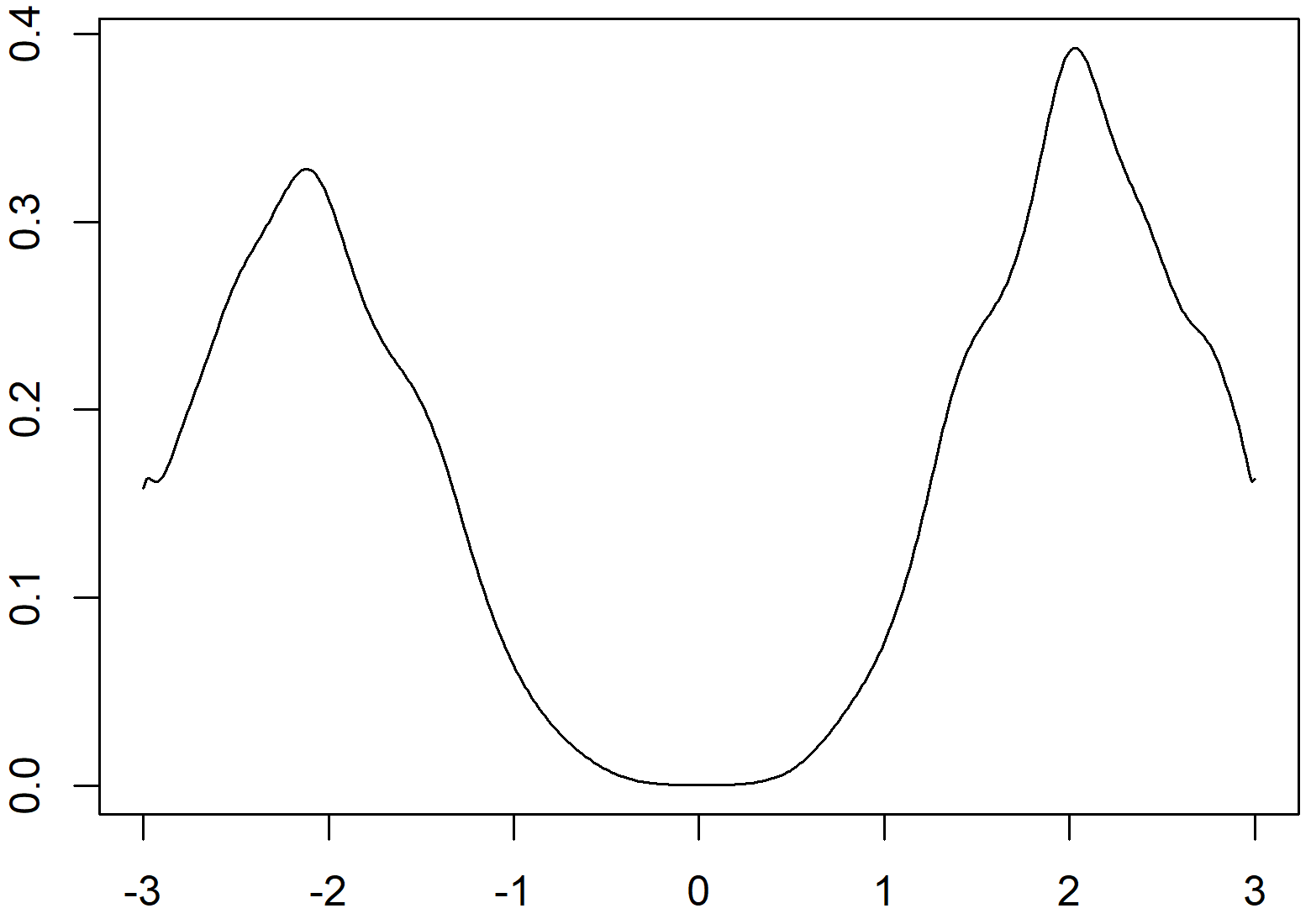}			
		\end{tabular}
		\end{figure}
	\end{example}
	
	\section{Statistical inference} \label{stat}
	In this section we provide a statistical procedure to estimate the attractor space, which is assumed here to be finite dimensional. 	\cite{Chang2016152} provided statistical methods based on FPCA for a  cointegrated density-valued linear process with values in $L^2(\lambda)$. Even though \cite{Beare2017} commented that their time series cannot accommodate a nontrivial attractor space, the statistical procedure they provided is useful since we can apply it to the time series of $\clr$-images $\tilde{f}=(\tilde{f}_t, t \geq 1)$ with values in $\overline{L^2}(\lambda)$. Using their procedure, the attractor space of $\overline{L^2}(\lambda)$ can be estimated, and then it is easy to obtain the attractor space of $B^2(\lambda)$ using the $\clr^{-1}$ transformation. We therefore do not have to develop a new statistical technique adapted to Bayes Hilbert spaces.
	
	\begin{assumption*}[T1] \hspace{0.1cm}
		\begin{itemize}
			\item[(i)] For some finite rank $\tilde{N} \in \mathfrak B(\overline{L^2}(\lambda))$ and a sequence $(\tilde{N}_k, k \geq 0)$  in  $\mathfrak B(\overline{L^2}(\lambda))$ such that $\sum_{k=0}^\infty k \|\tilde{N}_k\|_{\overline{L^2}(\lambda)} < \infty$, we have \begin{align*}
			\tilde{f}_t = \tilde{f}_0 - \tilde{\nu}_0 + \tilde{N} \tilde{\xi}_t +  \tilde{\nu}_t, \quad t \geq 0 .	\end{align*}
			where $\tilde{\nu}_t = \sum_{k=0}^\infty \tilde{N}_k \tilde{\varepsilon}_{t-k}$.
			\item[(ii)] $(\tilde{\varepsilon}_t, t\in \mathbb{Z})$ is an iid sequence with $E \tilde{\varepsilon}_t = 0$, positive definite covariance operator $C_{\tilde{\varepsilon}_0}$, and $E\|\tilde{\varepsilon}_t\|_{\overline{L^2}(\lambda)}^p < \infty$ for some $p\geq 4$.
		\end{itemize}
	\end{assumption*}
	\noindent In the above assumption, it is explicitly required that the attractor space $\mathfrak A(\tilde{f},\overline{L^2}(\lambda))$ is finite dimensional. An example of such an I(1) process is an AR($p$) process satisfying Assumption (N), which we considered in the previous section. Assume that $\dim(\mathfrak A(\tilde{f},\overline{L^2}(\lambda))) = r > 0$  and let $(v_i, i \in \mathbb{N})$ be the orthonormal basis of $\overline{L^2}(\lambda)$ such that 
	\begin{align*}
	&\mathfrak A(\tilde{f},\overline{L^2}(\lambda)) = \spn (v_1, \ldots, v_r), \\
	&\mathfrak  C(\tilde{f},\overline{L^2}(\lambda)) = \spn (v_{r+1}, v_{r+2},  \ldots).
	\end{align*}  
	
	The probability density functions $(f_t, t=1,\ldots,T)$ are not directly observed and must be estimated from the data. We assume that there are $n$ cross-sectional observations that are available to estimate $f_t$ at each time period $t$. One possibility would be to obtain estimated densities $(\hat{f}_1,\ldots,\hat{f}_T)$ from a standard nonparametric kernel method. However, since we would be applying the procedure of \cite{Chang2016152} to the $\clr$-images $\log \hat{f}_t - \int \log \hat{f}_t d \lambda$, this naive approach could lead to unacceptable bias if the densities are close to zero at the boundary of their support. Possibly, a better approach is to estimate the log-densities directly. Methods for doing so include the maximum penalized likelihood method of \cite{silverman1982}, the spline-based method of \cite{OSullivan1988} and the local likelihood method of \cite{loader1996}. However the clr-images are estimated, we assume that they satisfy the following high level condition.
	
	\begin{assumption*}[T2] \hspace{0.1cm}
		The estimated clr-images $(\tilde{f}^{\mathcal E}_t, t=1,\ldots,T)$ satisfy
		\begin{itemize}
			\item[(i)] $\sup_{1 \leq t \leq T} \|\tilde{f}^{\mathcal E}_t - \tilde{f}_t\|_{\overline{L^2}(\lambda)} = O_p(1)$,
			\item[(ii)] $T^{-1} \sum_{t=1}^T \| \tilde{f}^{\mathcal E}_t - \tilde{f}_t \|_{\overline{L^2}(\lambda)} \to_p 0$. 
		\end{itemize}	
	\end{assumption*} 

	The testing procedure is based on FPCA of the empirical covariance operator
	\begin{align*}
	\tilde{V}^T =  T^{-1}\sum_{t=1}^T \left(\tilde{f}^\mathcal{E}_t - T^{-1} \sum_{t=1}^T \tilde{f}^\mathcal{E}_t  \right) \otimes \left(\tilde{f}^\mathcal{E}_t - T^{-1} \sum_{t=1}^T \tilde{f}^\mathcal{E}_t  \right).
	\end{align*}
	Let $(\zeta_i^T, v_i^T)$, $i\in\mathbb N$, be the eigenpairs of the empirical covariance operator $\tilde{V}^T$, where $\zeta_i^T$ is assumed to be decreasing as $i$ gets larger.	Let $\mathbf{P}_{A}$ be the orthogonal projection on $\ran \tilde{N}$ and define $\mathbf{P}_{C} = \id_{\overline{L^2}(\lambda)} - \mathbf{P}_A$. $\tilde{V}^T$ can be written as
	\begin{align*}
	\tilde{V}^T = \tilde{V}_{AA}^T + \tilde{V}_{AC}^T +  \tilde{V}_{CA}^T +  \tilde{V}_{CC}^T,
	\end{align*}
	where $\tilde{V}_{AA}^T = \mathbf{P}_A \tilde{V}^T \mathbf{P}_A $, $\tilde{V}_{AC}^T = \mathbf{P}_A \tilde{V}^T \mathbf{P}_C $ , $\tilde{V}_{CA}^T = \mathbf{P}_C \tilde{V}^T \mathbf{P}_A, $ and $\tilde{V}_{CC}^T = \mathbf{P}_C \tilde{V}^T \mathbf{P}_C $.
	Under Assumptions (T1) and (T2), it can be shown that 
	\begin{align*}
	\| T^{-1} \tilde{V}^T - T^{-1}\tilde{V}_{AA}^T \|_{\mathfrak B(\overline{L^2}(\lambda))} = O_p(T^{-1}).
	\end{align*}
	That is, the difference between $T^{-1} \tilde{V}^T$ and $T^{-1} \tilde{V}_{AA}^T $ is a norm-vanishing operator. 
	Noting that the eigenfunctions of $\tilde{V}_{AA}^T$ are in $\mathfrak A(\tilde{f},\overline{L^2}(\lambda))$, it follows that the eigenfunctions associated with the $r$ leading eigenvalues are a consistent estimator of a spanning set of $\mathfrak A(\tilde{f},\overline{L^2}(\lambda))$. Formally, it can be shown \citep[Proposition 3.2 and Section 4]{Chang2016152} that
	\begin{align} \label{pconver}
	\| \hat{\mathbf{P}}_A - \mathbf{P}_A \|_{\overline{L^2}(\lambda)} = O_p(T^{-1}),
	\end{align}  
	where $\hat{\mathbf{P}}_A$ is the projection operator based on the eigenfunctions. An alternative formulation of \eqref{pconver} is
	\begin{align}
	\spn(v_1^T, \ldots, v_r^T) \to_p \mathfrak A(\tilde{f},\overline{L^2}(\lambda)).
	\end{align}
	Needless to say, $ \mathfrak C(\tilde{f},\overline{L^2}(\lambda))$ can be estimated as the orthogonal complement, $[\spn(v_1^T, \ldots, v_r^T)]^\perp$.  
	These results show that if $r$ is known, then it is easy to estimate the attractor space, $\mathfrak A(\tilde{f},\overline{L^2}(\lambda))$, since the $r$ leading eigenfunctions of $\tilde{V}^T$ asymptotically span $\mathfrak A(\tilde{f},\overline{L^2}(\lambda))$. 
	
	In practice, the dimension of $\mathfrak A(\tilde{f},\overline{L^2}(\lambda))$ is typically unknown. To determine the dimension, we may consider testing the null hypothesis
	\begin{align}
	&H_0 : \dim\left(\mathfrak A(\tilde{f},\overline{L^2}(\lambda))\right) = R \label{null},
	\end{align}
	{against the alternative}
	\begin{align}
	&H_1 : \dim\left(\mathfrak A(\tilde{f},\overline{L^2}(\lambda))\right) \leq R-1. \label{alter}
	\end{align}
	For some positive integer $R_{\max} \geq 1$,  successive tests of the null hypothesis for $R={R}_{\max}, {R}_{\max}-1, \ldots, 1$ can determine the dimension of $\mathfrak A(\tilde{f},\overline{L^2}(\lambda))$. For example, if the null is rejected in favor of the alternative for $R>r$, but not rejected for $R=r$ , we can conclude that $\dim(\mathfrak A\left(\tilde{f},\overline{L^2}(\lambda))\right)  = r$. A feasible test statistic and its limiting distribution under the null hypothesis are given by \cite{Chang2016152}. Let 
	\begin{align*}
	\hat{z}_t = \left( \left\langle v_1^T, \tilde{f}_t^{\mathcal E} - T^{-1} \sum_{t=1}^T \tilde{f}^\mathcal{E}_t  \right\rangle_{\mathcal H}, \ldots, \left\langle v_R^T, \tilde{f}_t^{\mathcal E} - T^{-1} \sum_{t=1}^T \tilde{f}^\mathcal{E}_t  \right\rangle_{\mathcal H} \right)'
	\end{align*}
	for $t=1,\ldots,T$, and define $\hat{Z}_T = (\hat{z}_1,\ldots, \hat{z}_T)'$. Further let  $\hat{Q}_R^T = \hat{Z}_T' \hat{Z}_T$ and $\hat{\Sigma}_R^T = \sum_{|k| \leq \ell} w_\ell (k) \hat \Gamma(k)$, where $\hat \Gamma(k)$ is the sample autocovariance of $\hat{z}_t$ and $w_{\ell}(k)$ is a bounded weight function\footnote{See \cite{Chang2016152} for further details. The truncation parameter $\ell$ may be chosen as in \cite{Andrews1991}. In the empirical applications reported in Section \ref{empirical} of this paper, $w_\ell(\cdot)$ is chosen to be the quadratic spectral kernel.}. The proposed test statistic is 
	\begin{align*}
	\tau_R^T = T^{-2} \zeta_{\min} (\hat{Q}_R^T, \hat{\Sigma}_R^T)
	\end{align*}
	where $\zeta_{\min} (\hat{Q}_R^T, \hat{\Sigma}_R^T)$ denotes the smallest generalized eigenvalue of $\hat{Q}_R^T$ with respect to $ \hat{\Sigma}_R^T$.
	\begin{proposition}[Theorem 4.3 in \citealp{Chang2016152}]
		If Assumptions  $(\mathrm{T1})$ and $(\mathrm{T2})$ are satisfied then, under the null hypothesis $H_0$ in \eqref{null}, we have
		\begin{align}\label{limdist}
		\tau_R^T \to_d \zeta_{\min} \left(\int \overline{W}_R (r) \overline{W}_R(r)' \mathrm{d}r , I_R\right),
		\end{align}
		where $I$ is the $R\times R$ identity matrix, $\overline{W}_R(r) = W_R(r) - \int W_R(s) \mathrm{d}s$, and $W_R$ is an $R\times1$ vector of independent standard Brownian motions. On the other hand, under the alternative hypothesis $H_1$ in \eqref{alter},  
		\begin{align*}
		\tau_R^T \to_p 0.
		\end{align*}
	\end{proposition}
	\noindent That is, $\tau_R^T$ converges in law to the smallest eigenvalue of $\int \overline{W}_R (r) \overline{W}_R(r)'\mathrm{d}r$ under the null hypothesis, and vanishes under the alternative. The critical values can be easily obtained from a large number of simulated sample paths of $\overline{W}_R$. The critical values are presented in Table \ref{tab1}.\footnote{The reported critical values for $R=1,\ldots,5$ are taken from Table 1 in \cite{Chang2016152}. The values for $R=6$ and $7$ are newly calculated for reference in the next section.} 	
	\begin{table}[h!]
		\centering
		\caption{Critical values of $\tau_R^T$}
		\label{tab1}
		\begin{tabular}{cccccccc}
			\hline
			R    & 1      & 2      & 3      & 4      & 5      & 6      & 7      \\\hline
			1\%  & 0.0248 & 0.0163 & 0.0123 & 0.0100 & 0.0084 & 0.0073 & 0.0065 \\\hline
			5\%  & 0.0365 & 0.0215 & 0.0156 & 0.0122 & 0.0101 & 0.0086 & 0.0075 \\\hline
			10\% & 0.0459 & 0.0254 & 0.0177 & 0.0136 & 0.0111 & 0.0094 & 0.0081
			\\\hline
		\end{tabular}\\
	\end{table}	
	
	\begin{remark}
		The case when $\tilde{f}_t$ is observable at each $t$ may be easily dealt with. Test statistics can be calculated based on true observations $\tilde{f} = (\tilde{f}_t, t=1,\ldots,T)$. The limit distribution is slightly different from \eqref{limdist}; it is given by 
		\begin{align*}
		\zeta_{\min} \left(\int W_R (r) W_R(r)' dr , I_R\right).
		\end{align*} See \cite{Chang2016152} for more details.  
	\end{remark}
	
	\section{Empirical application} \label{empirical}  
	\subsection{Example 1: Cross-sectional densities of earnings} 
	In this section, we revisit the application of FPCA to cross-sectional densities of individual weekly earnings undertaken by \cite{Chang2016152}. The cross-sectional observations that are used to estimate densities are obtained at monthly frequency from the Current Population Survey (CPS) database, running from January 1990 to March 2017 (327 months in total). Since individual earnings in each month are reported in current dollars, I adjusted them all to January 1990 prices.\footnote{Monthly CPI data are obtained from the Federal Reserve Economic Data (FRED).} I excluded earnings below the 2.5th percentile and above the 97.5th percentile since there are many near-zero\footnote{The dataset contains a lot of abnormal values of nominal earnings near zero. For example, there are total 9188 observations corresponding to weekly earnings less than $\$0.01$ over the whole span.} earnings and top-coded earnings in the raw dataset.\footnote{This symmetrical trimming was considered in, for example, \cite{autor2008} for analysis of the CPS wage data.} Weekly earnings are censored from above, with the threshold for censoring changing partway through the sample: the top-coded nominal earning is \$1923 before January 1998, and \$2885 afterward. This difference introduces significant heterogeneity; for example, the support of the earnings distribution changes greatly after January 1998 since the number of observations contained in $(1923, 2885)$ is always zero before this month, but nonzero afterward. However, after dropping the top 2.5\% of observations, the range of earnings is quite stable over time. Moreover, after dropping the bottom 2.5\% of observations, all near-zero earnings are excluded and the smallest earnings become reasonably sized. This exclusion of top and bottom 2.5\% of observations would also enhance the accuracy of the log-density estimation, which could have been significantly reduced by the scarcity of observations at the boundaries. After applying this truncation, the number of observations for each month ranges from 11760 to 15489. The last thing we need to notice is that the CPS individual earnings data are collected from a monthly survey of individuals, with each individual assigned their own design weight. Therefore, any estimates constructed from the dataset should take design weights into account. To estimate log-densities, I used the local likelihood method of \cite{Loader2006} which can easily accommodate design weights. The procedure is described in detail in the Appendix. Figure \ref{fig3} shows the time series of density estimates and their $\clr$-images.  
	
	To investigate the dimension of the attractor space, I calculated the test statistics described in the previous section. Figure \ref{tab2} displays the test statistics for $R=1,\ldots,7$ as well as a scree plot of the eigenvalues. From the scree plot it is apparent that $R_{\max}$ can be set to around $5$. Referring back to Table \ref{tab1}, we see that $R=2$ is rejected even at the $1\%$ level, but $R=1$ is not rejected at the $5\%$ level. We tentatively conclude that the dimension of the attractor space is one. Viewed as a subspace of $\overline{L}^2(\lambda)$, the estimated attractor space is the span of the leading eigenfunction $\hat{v}_1$, displayed in Figure \ref{tab2}. Alternatively, we may regard the attractor space to be the span of $\clr^{-1}(\hat{v}_1)$, a subspace of $B^2(\lambda)$.
	
	To describe the attractor space, I orthogonally projected the estimated densities upon the estimated cointegrating space, obtaining the projected series $(\hat{f}^C_t, t=1,\ldots,T)$. I positively and negatively perturbed the sample mean $\bar{f}^C_T$ of $(\hat{f}^C_t, t=1,\ldots,T)$, called the stationary mean of $f$, in the direction associated with the leading eigenfunction $\clr^{-1} \hat{v}_1$. More specifically, the positive perturbation is $\bar{f}^C_T\oplus(\hat{\zeta}_1 \odot\clr^{-1}(\hat{v}_1))$ and the negative perturbation is $\bar{f}^C_T\ominus(\hat{\zeta}_1 \odot \clr^{-1}(\hat{v}_1))$, where $\hat{\zeta}_1$ is the eigenvalue associated with $\hat{v}_1$. Figure \ref{figA1} shows the results. 
	\begin{figure}[h!]
		\centering
		\caption{Test statistics --- individual earnings}
		\label{tab2}
		\begin{tabular}{cccccccc}
			\hline
			$R$                        & 1      & 2      & 3      & 4      & 5      & 6      & 7      \\	\hline
			$	\tau_R^T$ & 0.03638 & 0.01253& 0.00908 & 0.00590 & 0.00537  & 0.00506 &  0.00472
		\end{tabular}
		\begin{tabular}{cc}\\
			Scree plot of eigenvalues & Leading eigenfunction\\
			\includegraphics[width=0.47\textwidth]{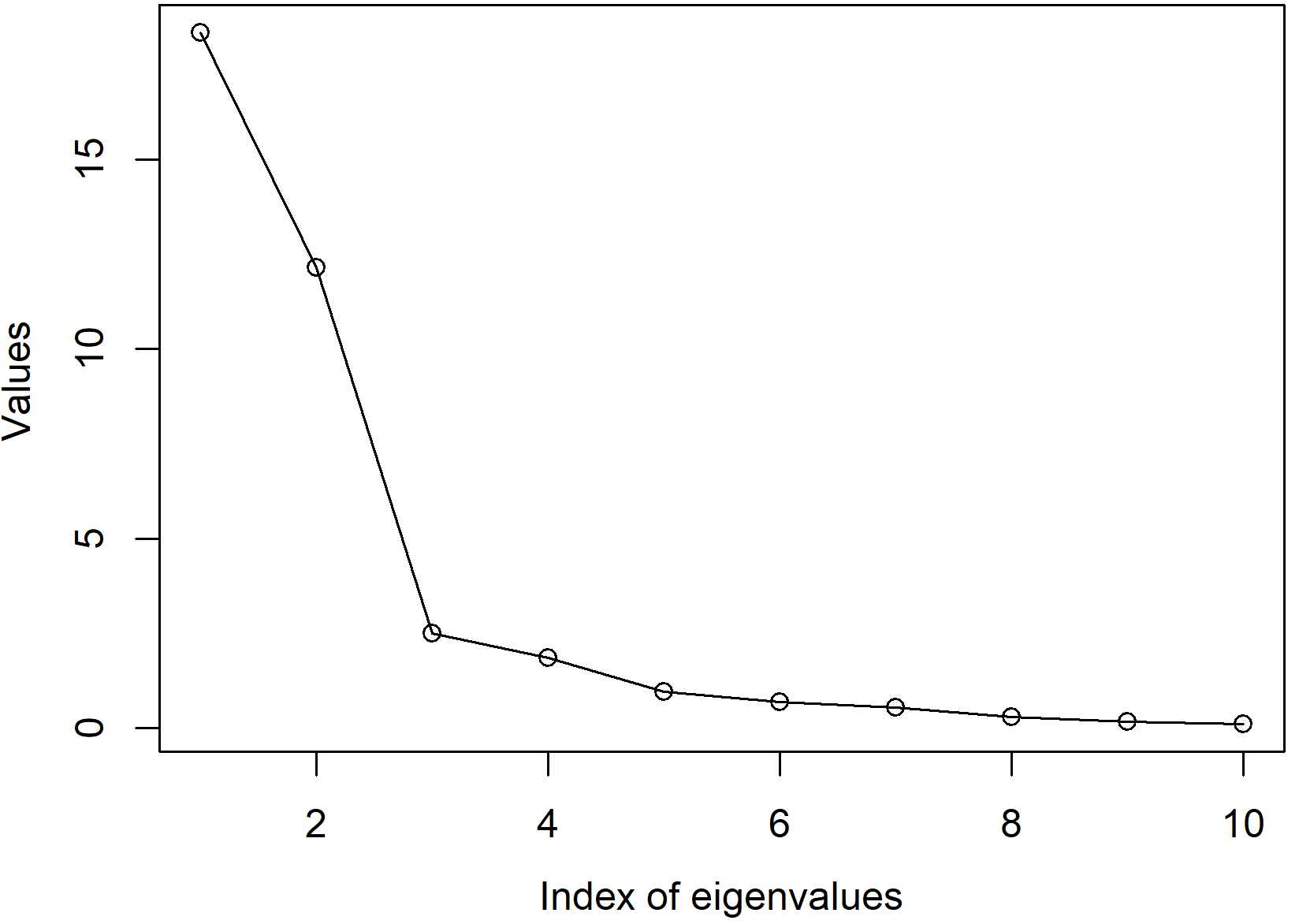} & 	\includegraphics[width=0.47\textwidth]{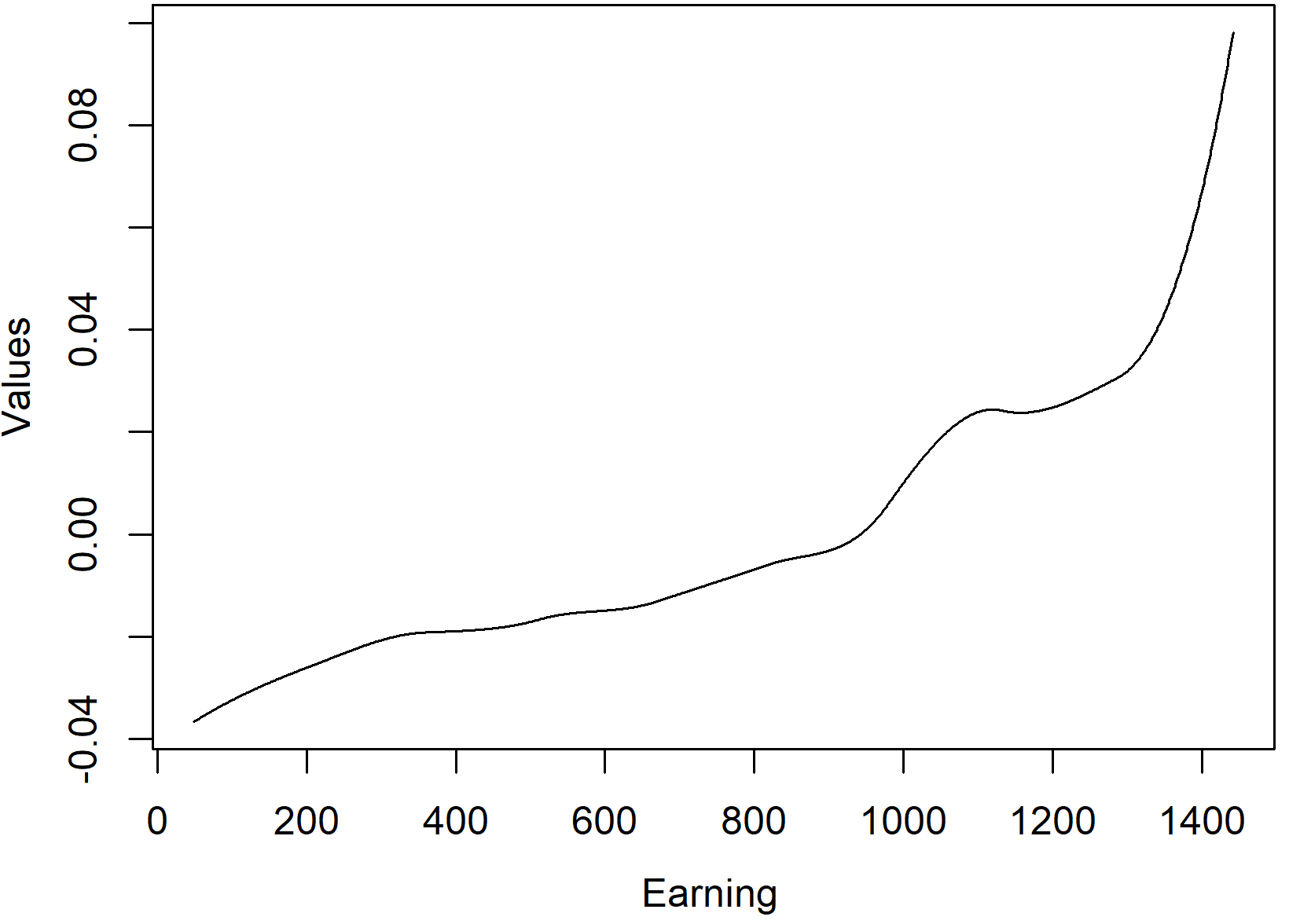}
		\end{tabular}		
	\end{figure}

	\begin{figure}[t!] 
		\caption{Density estimates and $\clr$-images - earnings} \label{fig3}
		\begin{tabular}{cc}\\
			Time series of density estimates & Time series of $\clr$-images\\
			\includegraphics[width=0.47\textwidth]{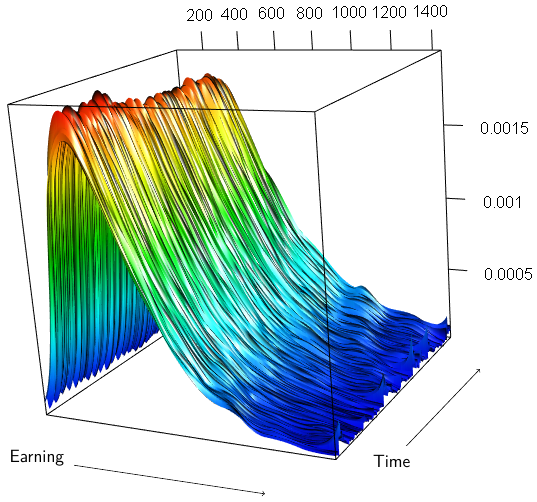} & 	\includegraphics[width=0.45\textwidth]{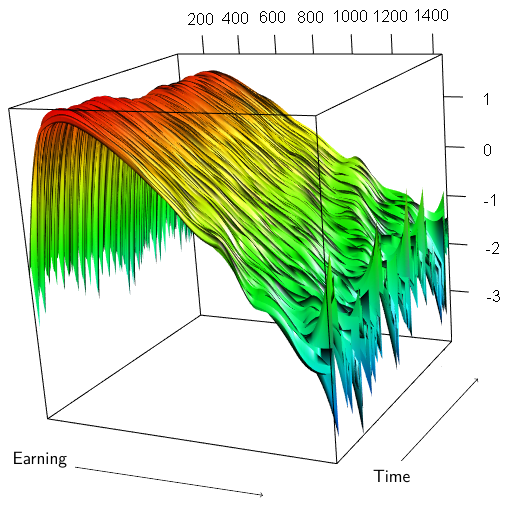}	
		\end{tabular}
	\end{figure}
	\begin{figure}[t!] 
		\centering
		\caption{Description of the attractor  - earnings} \label{figA1}
		\includegraphics[width=0.65\textwidth]{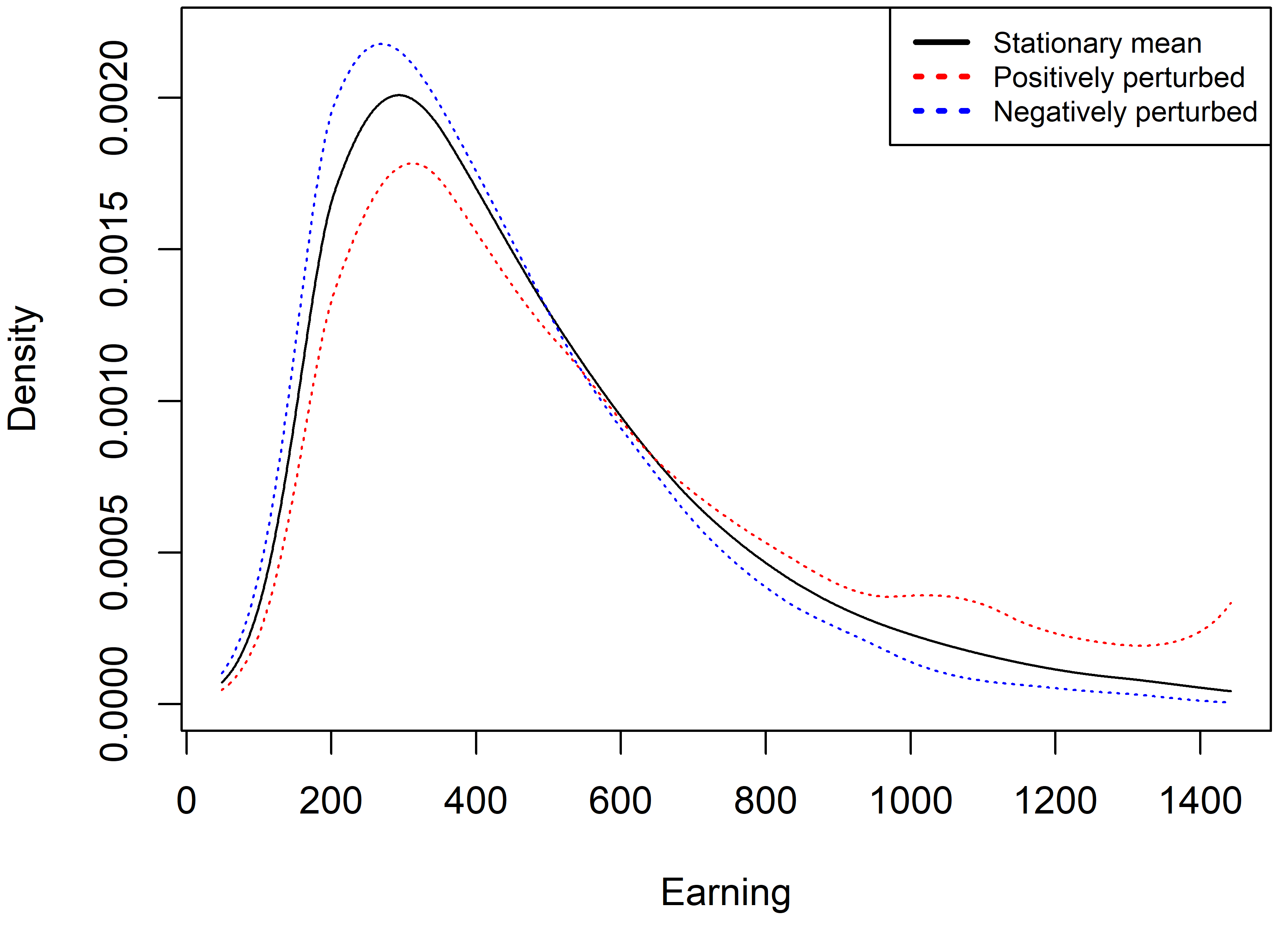} 	
	\end{figure}
	
	\subsection{Example 2 : cross-sectional densities of wages} 
	Now we consider cross-sectional densities of hourly wages. The cross-sectional observations are obtained at monthly frequency from the CPS database, and the span is the same as in the previous example. Wages are deflated using CPI data and measured in January 1990 dollars. There are many near-zero values and top-coded values in the raw data set, so I excluded earnings below the 1.25th percentile and above the 98.75th percentile; unlike the earnings data set, the top-coded values are constant over time, so only a small truncation is required to get rid of abnormal or extreme values. Similarly, this exclusion of top and bottom 1.25\% of observations would help enhance the accuracy of the log-density estimation at the boundaries.  The number of observations for each month ranges from 7348 to 9584  after truncation. Figure \ref{fig4} shows the time series of density estimates and their $\clr$-images.

	Test statistics are calculated as before and reported in Figure \ref{tab3} with other graphs. From the scree plot, we may set $R_{\max}$ to around $5$. From successive tests we conclude that the dimension of the attractor space is $1$.\footnote{On the other hand, if we set $R_{\max}$ equal to 6 (resp.\ 7) then our sequential testing procedure indicates an attractor space dimension of 6 (resp.\ 7), which seems large.} Thus, the estimated attractor space, viewed as a subspace of $\overline{L^2}(\lambda)$, is the span of the eigenfunction associated with the first leading eigenvalue. Viewed as a subspace of $B^2(\lambda)$, it is the span of the $\clr^{-1}$-image of this eigenfunction. 
	
	\begin{figure}[h!]
		\centering
		\caption{Test statistics - wages}
		\label{tab3}
		\begin{tabular}{cccccccc}
			\hline
			R                        & 1      & 2      & 3      & 4      & 5      & 6    &7      \\	\hline
			$	\tau_R^T$ & 0.05174 &0.01181  &0.01081  & 0.01076   &0.00982  &  0.00977 & 0.00966
		\end{tabular}
		\begin{tabular}{cc}\\
			Scree plot of eigenvalues &  Leading eigenfunction\\
			\includegraphics[width=0.47\textwidth]{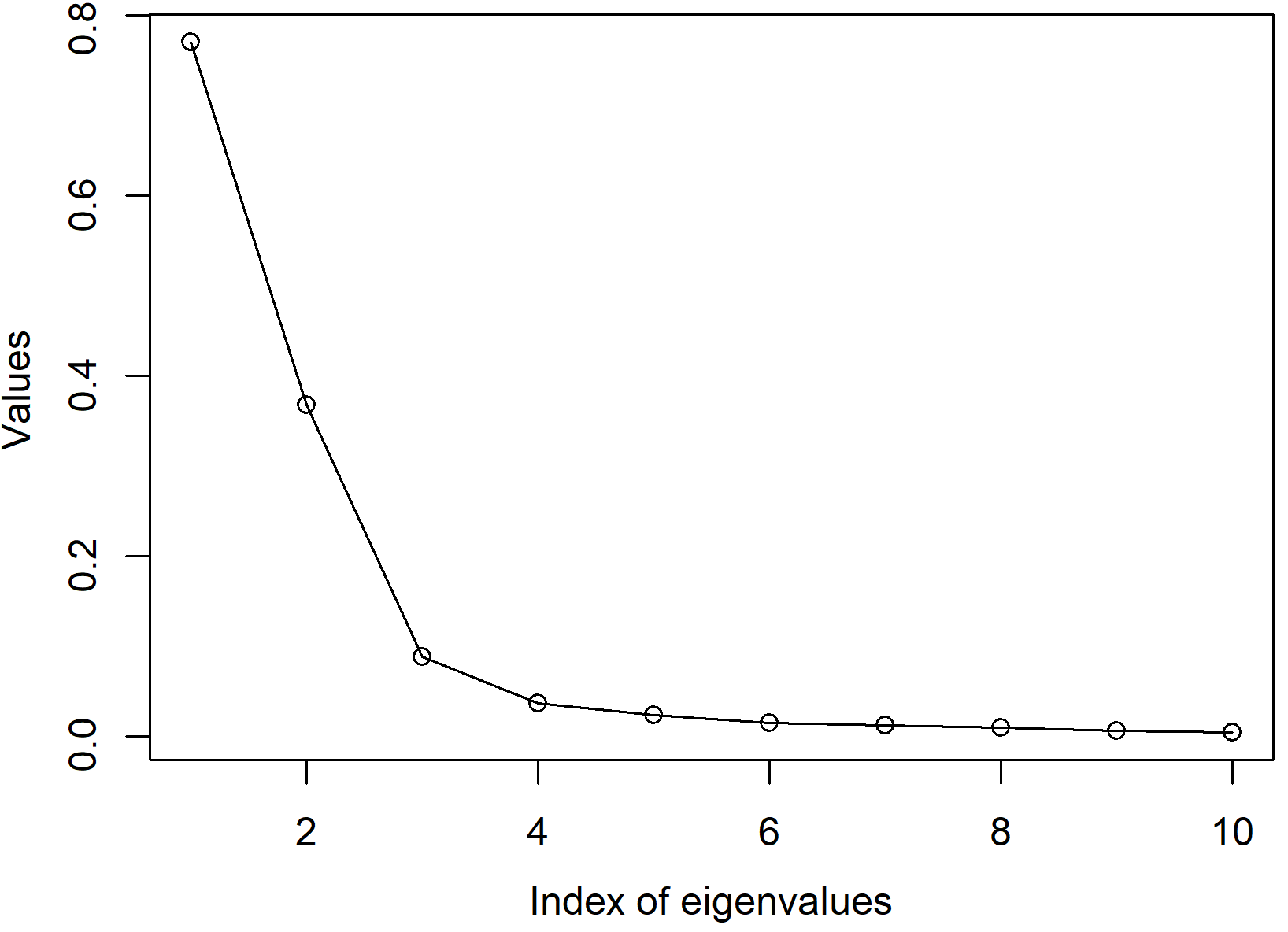} & 	\includegraphics[width=0.47\textwidth]{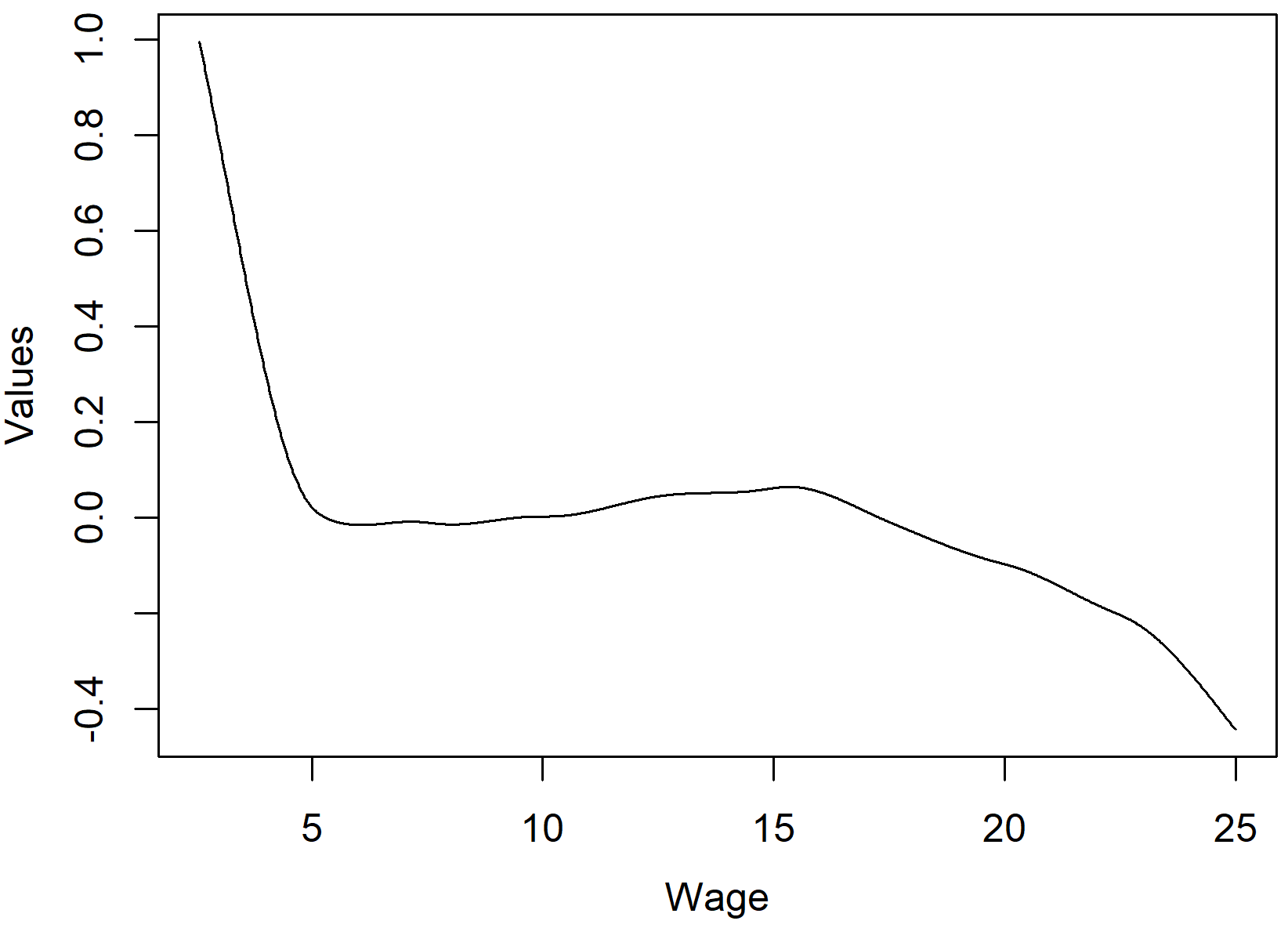}
		\end{tabular}
	\end{figure}

	Figure \ref{figA2} provides a visual impression of the attractor space, similar to Figure \ref{figA1} above. In this case I perturbed the stationary mean $\bar{f}^C_T$ in the direction of $\clr^{-1} \hat{v}_1$ by $\pm 2 \hat{\zeta}_1$, so as to more clearly emphasize the difference between the two perturbations.

	\begin{figure}[h!] 
		\caption{Density estimates and $\clr$-images - wages} \label{fig4}
		\begin{tabular}{cc}\\
			Time series of density estimates & Time series of $\clr$-images\\
			\includegraphics[width=0.47\textwidth]{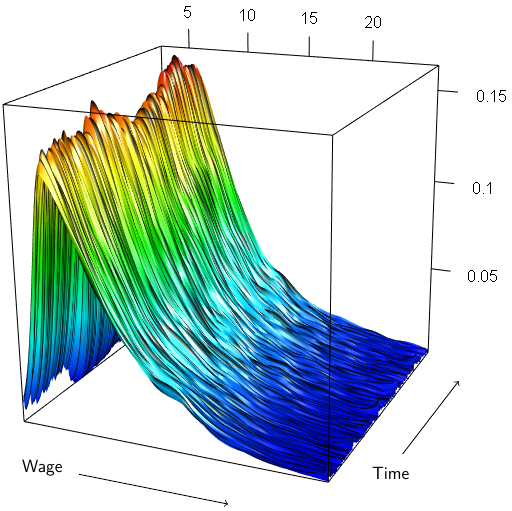} & 	\includegraphics[width=0.46\textwidth]{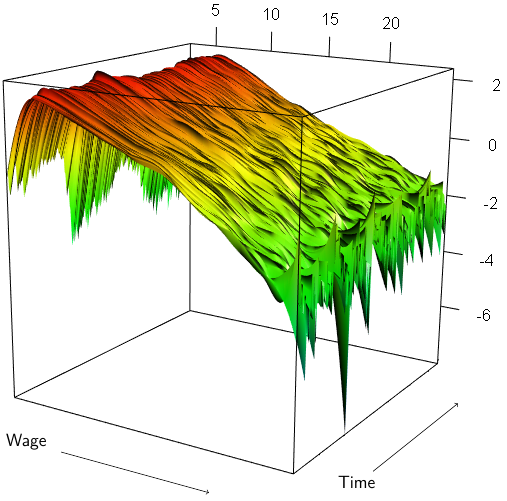}	
		\end{tabular}
	\end{figure}
	\begin{figure}[h!] 
		\centering
		\caption{Description of the attractor  - wages} \label{figA2}
		\includegraphics[width=0.65\textwidth]{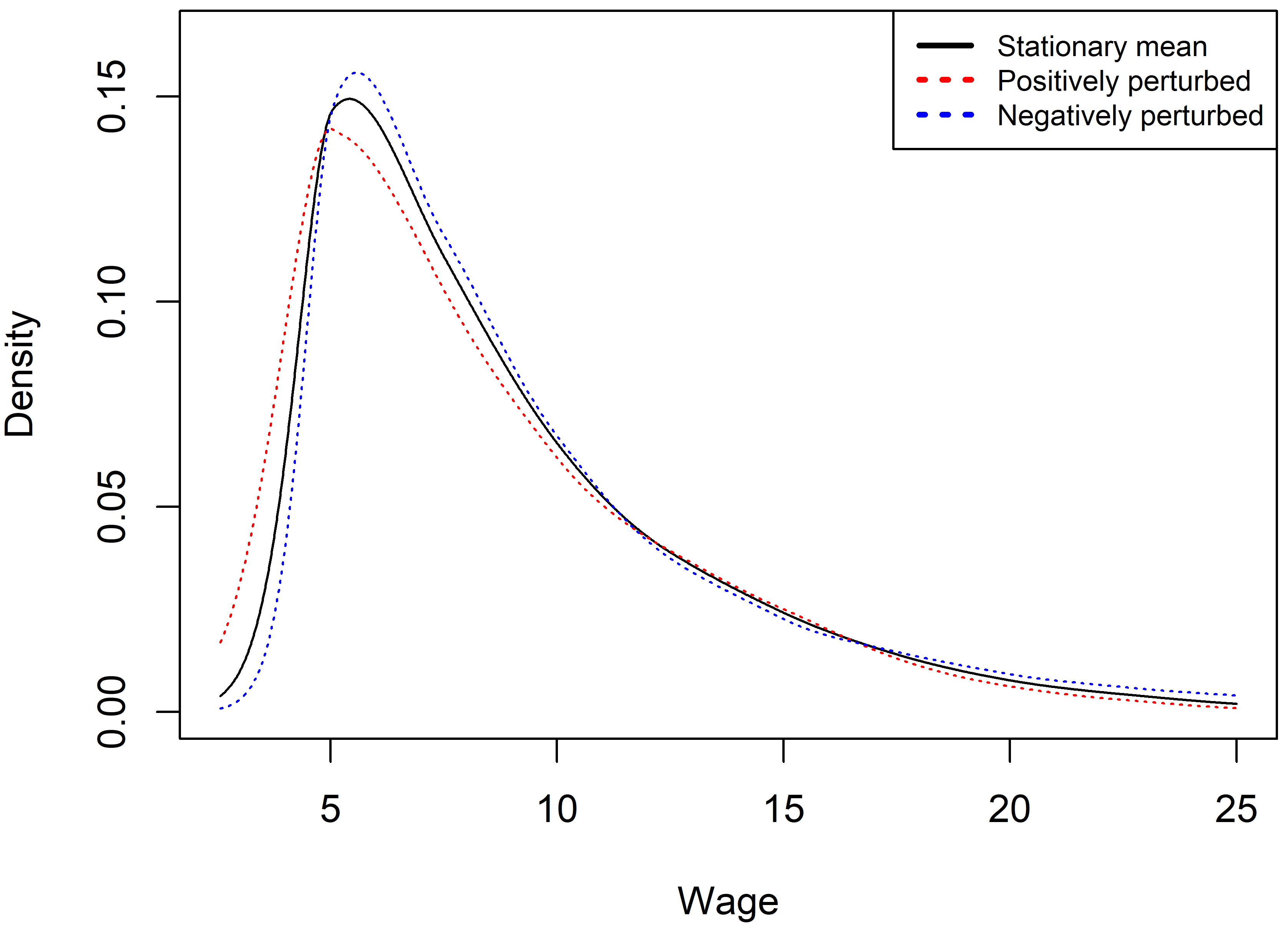} 	
	\end{figure}

	\section{Concluding remarks}
	In this paper we have investigated cointegrated linear processes with values in a Bayes Hilbert space of densities. Autoregressive density-valued processes were also studied and a version of the Granger-Johansen representation theorem is provided. We showed that the statistical methods developed by \cite{Chang2016152} can be used to estimate the attractor space associated with a cointegrated linear process in a Bayes Hilbert space. 
	\newpage
	\bibliographystyle{apalike}
	\bibliography{swkrefs}

	\newpage
	\section{Appendix}
	\subsection{Useful lemmas}
	\begin{lemma}[Theorem 1.4 in \cite{HOWLAND197112}]\label{lem1} Let $\mathfrak  B(\mathcal H)$ be the space of bounded linear operators $\mathcal H \to \mathcal H$ for a separable complex Hilbert space $\mathcal H$ and let $A(z) = \id_{\mathcal H} - C(z)$ where $C(z)$ is an analytic family of compact operators. If $z_0 \in \sigma(A)$ is an isolated element, there exist finite dimensional projections $Q_1,\ldots,Q_d$ such that 
		\begin{align*}
		A(z) = (P_1 + (z_0-z)Q_1) \cdots(P_d + (z_0-z)Q_d)G(z),
		\end{align*}
		where $G(z)$ is analytic, $G(z_0)$ is invertible, and $P_i = \id_H - Q_i$. We may choose $P_1$ to be any projection on $\ran A(z_0)$.
	\end{lemma}
	
	\begin{lemma}\label{lem2} Under the conditions of Lemma \ref{lem1}, 
		\begin{align*}
		(\id_H-P_1) A^{(1)}(z_0){\mid_{\ker A(z_0)}} : \ker A(z_0) \to \ran (\id_H-P_1) 
		\end{align*}
		is invertible if and only if $A(z)^{-1}$ has a simple pole at $z=z_0$.
	\end{lemma}
	\begin{proof}
		From Corollary 3.5 in \cite{HOWLAND197112}
		$A(z)^{-1}$ has a simple pole at $z=z_0$ if and only if (i) 	$A^{(1)}(z_0)$ is injective on $\ker A(z_0)$ and (ii) $\ran A(z_0) \cap A^{(1)}(z_0)\ker A(z_0) = \{0\}$.
		Since $A(z) = \id_{\mathcal H} - C(z)$ and $C(z)$ is a compact family of operators, $A(z)$ is an index-zero Fredholm family of operators. This implies that $\dim(\ker A(z_0)) = \dim(\ran (\id_H-P_1))$. From this, one can easily show that the invertibility of $(\id_H-P_1) A^{(1)}(z_0){\mid_{\ker A(z_0)}}$ is another equivalent condition to (i) and (ii).
	\end{proof}

	\begin{lemma}\label{lem3}
		Let $d=1$ in Lemma \ref{lem1}. Then,
		\begin{itemize}
			\item[\upshape{(a)}] $G(z_0) = A(z_0) - (\id_H-P_1) A^{(1)}(z_0)$,
			\item[\upshape{(b)}] $\ran G(z_0) \mid_{\ker A(z_0)} = \ran (\id_H - P_1)$,
		\end{itemize}
	\end{lemma}
	\begin{proof}
		 
		We may easily deduce (b) from (a) and Lemma \ref{lem2}, so we only show (a). We know that $G(z)$ is holomorphic at $z=z_0$ from Lemma \ref{lem1}, so
		\begin{align*}
		G(z)= G(z_0) - G^{(1)}(z_0) (z_0-z) + \sum_{j=2}^\infty \frac{G^{(j)}(z_0)}{j!}(z-z_0)^j.
		\end{align*}
		Therefore,
		\begin{align} \label{lemeq1}
		A(z) = [P_1 + (\id_H-P_1 )(z_0-z)][G(z_0)- G^{(1)}(z_0) (z_0-z) + H_1(z) ],
		\end{align}	
		where $H_1(z)$ denotes the remainder of the Taylor series of $G(z)$.  
		Since $A(z)$ is holomorphic at $z=z_0$, we have
		\begin{align}\label{lemeq2}
		A(z) = A(z_0) - A^{(1)}(z_0) (z_0-z) + H_2(z),
		\end{align}
		where $H_2(z)$ is the remainder of the Taylor series of $A(z)$.
		Collecting terms associated with the same powers from \eqref{lemeq1} and \eqref{lemeq2} we obtain
		\begin{align}
		&P_1 G(z_0) = A(z_0) \label{lemeq22}\\ 
		&(\id_H -P_1) G(z_0) = - A^{(1)}(z_0) + P_1 G^{(1)}(z_0).
		\end{align}
		The left-hand side of the second equation is invariant under $(\id_H - P_1)$, implying that 
		\begin{align}\label{lemeq3}
		(\id_H -P_1) G(z_0) = - (\id_H -P_1)  A^{(1)}(z_0).
		\end{align}
		Combining \eqref{lemeq22} and \eqref{lemeq3}, we obtain (a).
	\end{proof}

	\subsection{Log-density estimation}
	In Section \ref{empirical}, the estimated log-density $\hat{g}_t$ is obtained from the following procedure. The reader is referred to \cite{loader1996,Loader2006} for more details. 
	
	Given survey responses $X_1, \ldots, X_n$ with design weights $w_1, \ldots, w_n$ such that $\sum_{i=1} w_i = n$, consider the weighted log-likelihood
	\begin{align*}
	\mathcal L(f_t) = \sum_{i=1}^n w_i \log (f_t(X_i)) - n \left(\int f_t(u) du - 1 \right).
	\end{align*} 
	Let $K$ be the support of $f_t$. Under some local smoothness assumptions, we can consider a localized version of the log-likelihood and $\log f_t(u)$ can be locally approximated by a polynomial function, so follows.
	\begin{multline} \label{locallf}
	\mathcal L_p(f_t)(x) = \sum_{i=1}^n w_i  \mathcal W\left(\frac{X_i-x}{h}\right) \mathcal Q(X_i-x ; \alpha_t)\\ - n \int \mathcal W\left(\frac{u-x}{h}\right) \exp(\mathcal Q(u-x ; \alpha_t)) du
	\end{multline} 
	where  $\mathcal W$ is a suitable kernel function,  $h$ is a bandwidth which assumed to be fixed, and $\mathcal Q(u ; \alpha_t)$ is polynomial in $u$ with coefficients $\alpha_t$. 
	
	I set $\mathcal W(u) = \frac{70}{81}(1-|u|^3)^3$ and $\mathcal Q(u ; \alpha) = \alpha_{0,t}  + \alpha_{1,t} u$. For fixed $x \in K$, let $(\hat{\alpha}_{0,t},\hat{\alpha}_{1,t})$ be the maximizer of \eqref{locallf}. Then the local likelihood log-density estimate is given by
	\begin{align*}
	\hat{g}_t (x) = \hat{\alpha}_{0,t}
	\end{align*}
	The procedure is repeated for a fine grid of points, and then $\hat{g}_t$ may be obtained from an interpolation method  described in  \citep[Chapter 12]{Loader2006}. 
	Needless to say, the above estimation procedure depends on $h$, a fixed bandwidth parameter. Let $b_t \coloneqq \left[q_t(99)-q_t(1)\right]n_t^{-1/5}$ where $q_t(a)$ is the $a$-th percentile of observations at time $t$ and $n_t$ is the number of observations. Within the range $[0.75 \, b_t,  1.25\, b_t ]$, bandwidth $h$ is set to the minimizer of a generalized BIC criterion.\footnote{A generalized version of AIC as a diagnostic for the local likelihood method is given in \cite{Loader2006}. A suitable generalization of BIC is obtained by changing the penalty term in an obvious way.}

	
	%
	
	
\end{document}